\documentclass[10pt]{amsart}

\usepackage{amscd,amssymb,verbatim,url,amsfonts,amssymb,amsthm,wrapfig}

\theoremstyle{plain}
\newtheorem{thm}{Theorem}[section] 
\newtheorem{lem}[thm]{Lemma}
\newtheorem{cor}[thm]{Corollary}

\newtheorem{prop}[thm]{Proposition}

\theoremstyle{definition}

\newtheorem{defn}[thm]{Definition}

\newtheorem{conj}[thm]{Conjecture}
\newtheorem{ques}[thm]{Question}
\newtheorem{rem}[thm]{Remark}

\numberwithin{equation}{section}
\numberwithin{thm}{section}

\newcommand{\olambda}{{\overset\circ\lambda}}
 \newcommand{\rat}{{\text{\rm rat}}}
\newcommand{\oPplus}{\overset{\circ \hspace*{5pt}}{P^+}}
\newcommand{\oPhiplus}{\overset{\circ \hspace*{5pt}}{\Phi^+}}
\newcommand{\fhstar}{\frak{h}^{*+}}

\newcommand{\ofhstar}{\overset{\circ \hspace*{6pt}}{\frak{h}^{*}}}\newcommand{\nat}{{\text{\rm nat}}}
\newcommand{\bM}{{\mathbb M}}
\newcommand{\fh}{{\mathfrak h}}
\newcommand{\wh}{{\widetilde{\mathfrak h}}}
\newcommand{\oh}{{\overset\circ{\mathfrak h}}}
\newcommand{\fg}{{\mathfrak g}}
\newcommand{\og}{{\overset\circ{\mathfrak g}}}

\newcommand{\g}{{\mathfrak g}}
 \newcommand{\widep}{\widetilde{\mathfrak p}}

\newcommand{\oP}{{\overset\circ P}}
\newcommand{\ohstar}{\overset\circ{\mathfrak h}^*}
\newcommand{\oPhi}{{\overset\circ\Phi}}
\newcommand{\oPi}{{\overset\circ\Pi}}

\newcommand{\bO}{{\mathbb O}}
\newcommand{\wO}{{\widetilde{\mathbb O}}}

\newcommand{\orho}{{\overset\circ\rho}}

\newcommand{\Mres}{{\mathbb M}^{\text{\rm res}}}

\newcommand{\nilrad}{t{\mathbb C}[t]\otimes \g}
\newcommand{\q}{{\bold q}}
\newcommand{\wleq}{{\,\overset\sim\leq\,}}

\newcommand{\wH}{{\widetilde H}}

\newcommand{\sZ}{{\mathcal{Z}}}

\newcommand{\sC}{{\mathcal {C}}}
\newcommand{\wg}{{\widetilde{\mathfrak g}}}

 \newcommand{\omu}{{\overset\circ\mu}}

\newcommand{\Ext}{{\text{\rm Ext}}}

\newcommand{\gr}{\operatorname{ {gr}}}

\newcommand{\ad}{\operatorname{ad }}

\newcommand{\sL}{{\mathcal L}}
\newcommand{\Hom}{\text{\rm Hom}}

\newcommand{\End}{\operatorname{End}}

\newcommand{\sO}{{\mathcal{O}}}
\newcommand{\rad}{\operatorname{rad}}

\newcommand{\wM}{{\widetilde{M}}}

\newcommand{\wN}{{\widetilde{N}}}

\newcommand{\wS}{{\widetilde{S}}}
\newcommand{\wT}{{\widetilde{T}}}
\newcommand{\wL}{{\widetilde{L}}}
\newcommand{\sE}{{\mathcal{E}}}

\newcommand{\go}{\overset{\circ}{\mathfrak g}}

\newcommand{\blist}{\begin{list}{\rom{(\roman{enumi})}}{\setlength
{\leftmarg in}{0em} \setlength{\itemindent}{7ex}
\setlength{\labelsep}{2ex}\setlength{\listparindent}{\parindent}
\usecounter{enumi}}}
\newcommand{\elist}{\end{list}}

\begin{document}

\begin{abstract} In \cite{PS6}, the authors studied the radical filtration of a Weyl module $\Delta_\zeta(\lambda)$
for quantum enveloping algebras $U_\zeta(\overset\circ{\mathfrak g})$ associated to a finite dimensional complex semisimple
Lie algebra $\overset\circ{\mathfrak g}$.  There $\zeta^2=\sqrt[e]{1}$ and $\lambda$ was, initially, required to be $e$-regular.  Some additional restrictions on $e$ were required---e.~g., $e>h$, the Coxeter
number, and $e$ odd.
  Translation to a facet gave an explicit semisimple series for all
quantum Weyl modules with singular, as well as regular, weights. That is, the sections of the filtration are  explicit semisimple modules with computable multiplicities of irreducible constituents. However, in the singular case, the filtration
conceivably might not be the radical filtration. This paper shows how a similar
semisimple series result can be obtained for all positive integers $e$ in case $\overset\circ{\mathfrak g}$
has type $A$, and for all positive integes $e\geq 3$ in type $D$.   One application describes 
semisimple series (with computable multiplicities) on $q$-Specht modules. 
We also discuss an analogue for Weyl modules for classical Schur algebras and Specht modules for
symmetric group algebras in positive characteristic
$p$. Here we assume the James Conjecture and a version of the Bipartite Conjecture. 
\end{abstract}

 \title[ A semisimple series for $q$-Weyl and $q$-Specht modules]{A semisimple series for $q$-Weyl and
 $q$-Specht modules}
\author{Brian J. Parshall}
\address{Department of Mathematics \\
University of Virginia\\
Charlottesville, VA 22903} \email{bjp8w@virginia.edu {\text{\rm
(Parshall)}}}
\author{Leonard L. Scott}
\address{Department of Mathematics \\
University of Virginia\\
Charlottesville, VA 22903} \email{lls2l@virginia.edu {\text{\rm
(Scott)}}}

\thanks{Research supported in part by the National Science
Foundation}


\medskip
\medskip
\medskip \maketitle

\section{Introduction} In the modular representation theory of a reductive group $G$ (or a quantum
enveloping algebra $U_\zeta(\overset\circ{\mathfrak g})$, with $\zeta^2$ a primitive
$e$th root of 1), the general failure of complete reducibility
has given rise, in the past 40 years, to a rich cohomology theory for both $G$ and $U_\zeta(\overset\circ{{\mathfrak g}})$. See \cite{JanB} for a compilation  of many results.  The related question of better understanding  important
filtrations of certain modules, e.~g., Weyl modules, also has attracted considerable attention. See, for
example, \cite[pp. 445, 455]{Jantzenpaper},   \cite[\S8]{AK}, \cite{Pillen}, \cite{PS6} on filtrations with semisimple
sections  as well as \cite[\S3]{J2}, \cite{Don},  \cite{A1}, \cite[\S6]{CPS9}, and
\cite{PS7} for the somewhat analogous $p$-filtrations.

Interesting filtrations can take many forms, but a basic filtration for any finite dimensional
module $M$ is its radical filtration
$M\supseteq \rad M\supseteq \rad^2M\supseteq \cdots$. In this case, the sections
$\rad^iM/\rad^{i+1}M$ are, of course, semisimple (i.~e., completely reducible), so that $\{\rad^iM\}$ is an example
of a ``semisimple series,"  mentioned in the title of this paper. 
 In recent work, the authors \cite{PS6} succeeded
in calculating the multiplicities of the irreducible constituents for the radical series sections in the quantum Weyl modules associated to regular weights. It was required that $e>h$, the Coxeter number of $\mathfrak g$. In addition, $e$ was required to be odd (and there were some other mild conditions on $e$, depending on the
root system). For
such ``large" $e$, we also could describe the sections in a semisimple series for quantum Weyl modules
with singular highest weights (but we were unable to
show the series was the radical series, though this seems likely to be the usual case). Our methods also were applicable for Weyl modules
in sufficiently large positive characteristics having highest weights in the Janzten region.  

This paper completes part of this project by giving, for types $A$ and $D$, an explicit semisimple series for quantum
Weyl modules for {\it all} positive integers $e$, except that in type $D_{2m+1}$ it is  required $e\geq 3$.  
Explicit formulas for the multiplicities of the irreducible modules for each semisimple
section are also obtained.   In particular, in type $A$, our previous results are extended to all small $e$ and all even $e$.  Interestingly, these previous results, given in \cite{PS6} for $e$ odd and $>h$, play a key role here in obtaining the results for
$e$ even and/or small.  Extensions of these results to other types would be possible provided there were improvements in
the Kazhdan-Lusztig correspondence as quoted in \cite[p. 273]{T}. This paper is organized so as to
make such extensions easy to obtain once such improvements are known. 

The method involves passing, for suitably  large $e$, to an equivalent category of modules for the
(untwisted) affine Lie algebra $\mathfrak g$ attached to $\overset\circ{\mathfrak g}$.  Category equivalences 
at the affine Lie algebra level provides the flexibility to treat small/even values of $e$ (and then
pass back to the quantum case, using the work \cite{KL} of Kazhdan-Lusztig and our own
results \cite{PS6}). Our approach
is non-trivial and takes up \S\S3--7. It requires the interaction of several highest weight categories of Lie
algebra modules  (some of them
new) and exact functors between them.  In particular, we treat (various versions of) categories of
$\widetilde{\mathfrak g}=[{\mathfrak g},{\mathfrak g}]$-modules which are integrable in the direction of $\overset\circ{\mathfrak g}$, and we also study their associated
standard and costandard modules. Section 7 contains several contributions to further understanding these categories; see,
for instance, Theorem \ref{theorem7.3} which both mirrors and uses the filtration results of \cite[Thms. 8.4, Cor. 8.5]{PS6}, and whose
proof requires the combinatorial equivalences obtained by Feibig 
\cite[Thm. 11]{Fiebig}. All of this work is done when
  $\overset\circ{\mathfrak g}$ is an arbitrary complex semisimple Lie algebra. Much of what we need for
  the quantum case (in particular, the entire $e$ odd case) could be done by working with the translation functor theory we provide, which
  gives many categorial equivalences without the need to construct inverses at a Verma flag level, as
   in \cite{Fiebig0}, or to construct explicit combinatorial deformations, as in \cite{Fiebig}. However, the latter theory of Fiebig is theoretically very satisfying and  has many additional practical advantages. In particular, 
   it allows us, in our quantum situation, to deal with the $e$ even case.

 One huge advantage of our extension of the results of \cite{PS6} to small $e$ is that the results can be
used to obtain, in type $A$ and working with the $q$-Schur algebras $S_q(n,r)$ with $q=\zeta^2$, semisimple series and multiplicity formulas for the Specht modules of the Hecke
algebras $H_q(r)$. Small $e$ results are required because the contravariant Schur functor from $S_q(n,r){\text{\rm -mod}}
\to{\text{\rm mod-}}H_q(r)$, taking Weyl modules to Specht modules, is only
exact when $r\leq n$.  On the other hand, the treatment of meaningful cases (i.~e., $H_q(r)$ not semisimple) requires
$e\leq r$, so that $e\leq n=h$. Thus, $e$ is ``small" in the sense of this paper. (Also, except when $e=r=h$, we have $e<h$, and all weights are singular.)

Another application is to Weyl modules for classical Schur algebras $S(n,r)$ in characteristic $p>0$. The
weights $\lambda$ are required to be viewed as partitions of a positive integer $r$ satisfying
$r<p^2$. Also, we assume (the defining characteristic version of) the James conjecture \cite{James} and a Schur algebra version of the Bipartite Conjecture \cite{FT}; see \S8.3. With these assumptions, we show that both Weyl modules and corresponding
Specht modules have explicit semisimple series, with multiplicities of irreducible modules explicitly given in
terms of inverse Kazhdan-Lusztig polynomials. 

Returning to the quantum case, there is an interesting overlap, in type $A$ with $e\not=3$, between our results and methods and those of Peng Shan \cite{Shan}. Her focus is on the Jantzen filtration and ours is on a semisimple
series. In the case of regular weights, the sections of the Jantzen filtration are semisimple; in fact, Theorem \ref{quantum}(a), together with the multiplicities given in \cite{Shan}, imply, in the regular weight case, that the Jantzen filtration is the radical filtration. Semisimplicity
of the sections of the Jantzen filtration 
remains unknown
for singular weights.  However, semisimplicity is likely, since the section multiplicities in \cite{Shan} agree with
those for the semisimple series studied in this paper.

  In \S9, Appendix I, we provide (apparently new) equivalences in the affine case between 
$\uparrow$-style orders \cite{JanB}, \cite{KK}, and the Bruhat-Chevalley order. The proofs in this section
are all combinatorial. The results are used in our proofs here, and Theorem \ref{theorem9.6} has also been used
in \cite{HPS} to complete an argument in \cite{ABG}, relevant to the Koszulity of some of the
algebras $A$ we consider in the regular case. See footnote 7. In Theorem \ref{theorem7.3}, for example, we
prove only that $\gr A_\Gamma$ is Koszul, not the stronger property that $A_\Gamma$ is Koszul. Although the
Koszulity of $\gr A_\Gamma$ is all that is needed in the semisimple series results in this paper, it is still 
interesting to know about the Koszulity of $A_\Gamma$, as argued in footnote 7.  For a (non-Lie theoretic) example when
$\gr A$ is Koszul, but $A$ is not Koszul, see [11].

\section{Notation: Lie algebras} The following notation is standard, mostly following \cite{Kac},
\cite{KT1}, and \cite{KT2} with cosmetic differences. (For example, our root system is denoted $\Phi$
rather than $\Delta$. The classical finite root system is denoted $\oPhi$, and the maximal
long and short roots of $\oPhi$ are denoted $\theta_l$ and $\theta_s$, respectively.)  
If $V$ is a complex vector space and $V^*$ is its dual, the natural pairing $V^*\times V\to\mathbb C$ between $V^*$ and $V$ is
usually denoted $\langle \phi,v\rangle= \phi(v)$. Unexplained notation is {\it very} standard.

\medskip
\noindent\underline{Finite notation:}
\begin{enumerate}

\item $\og$ is finite dimensional, complex, simple Lie algebra, with Cartan subalgebra $\oh$, Borel subalgebra
${\overset\circ{\mathfrak b}}\supset\oh$.

\item $\oPhi$, $\oPi=\{\alpha_1,\cdots,\alpha_r\}$, $\oPhiplus$: roots of $\oh$ in $\og$, simple
roots determined by $\mathfrak b$, positive roots determined by $\oPi$. These are subsets of
$\ofhstar$, which is identified with $\oh$ using the restriction of the Killing form on $\og$ to $\oh$
normalized so that the induced form on $\ofhstar$ satisfies $(\theta_l,\theta_l)=2$ if $\theta_l$
is the maximal root in $\oPhi$. Let $\orho=\frac{1}{2}\sum_{\alpha\in\oPhiplus}\alpha$ be the Weyl weight.

\item $\oPhi^\vee=\{\alpha^\vee\,|\,\alpha\in\oPhi\}$: coroot system of $\oPhi$, identifies with
 $\{\alpha^\vee=2\alpha/(\alpha,\alpha)\,|\,\alpha\in\oPhi\}$. 

\item $\overset\circ W=\langle s_{\alpha_2}, \cdots, s_{\alpha_r}\rangle$: Weyl group of $\og$, generated
by fundamental reflections $s_{\alpha_i}:{\mathbb E}\to{\mathbb E}$, where ${\mathbb E}=\ofhstar_{\mathbb R}$ is the Euclidean space associated to the Killing form on $\oh$.

\item $\varpi_1,\cdots,\varpi_r$: fundamental dominant weights; thus, $(\varpi_i,\alpha_j^\vee)=\delta_{i,j},$
$1\leq i,j\leq r$.

\item$\overset\circ P$, $\oPplus$: the weight lattice $\bigoplus_{i=1}^r{\mathbb Z}\varpi_i$, set
$\bigoplus_{i=1}^r{\mathbb N}\varpi_i$ of dominant weights.

\item $h$, $g$: Coxeter and dual Coxeter numbers; thus $h-1=(\orho,\theta^\vee_s)$ and $g-1=\langle\orho,
\theta_l^\vee\rangle$, $i=1,\cdots, r$. 
 \end{enumerate}

 \medskip
 \noindent
 \underline{Affine notation:}
 \begin{enumerate}

\item ${\mathfrak g}:=({\mathbb C}[t,t^{-1}]\otimes\go)\oplus {\mathbb C}c\oplus{\mathbb C}d$: affine Lie algebra
attached to $\og$, with central element $c$.

\item ${\mathfrak h}={\overset\circ{\mathfrak h}}\oplus {\mathbb C}c\oplus
{\mathbb C}d$ be the ``Cartan subalgebra" of $\mathfrak g$. Following \cite[p. 268]{T}, consider
$\chi,\delta\in{\mathfrak h}^*$ defined by 
$$\begin{cases}\chi({\mathfrak h})=\delta({\mathfrak h})=0;\\ \chi(c)=\delta(d)=1;\\
\chi(d)=\delta(c)=0.\end{cases}$$
 Thus, 
\begin{equation}\label{decomp}{\mathfrak h}^*:=\overset\circ{\mathfrak h}^* \oplus{\mathbb C} \chi\oplus{\mathbb C}\delta.\end{equation}
Here $\ofhstar$ identifies with a subspace of ${\mathfrak h}^*$ by making
it vanish on $d$ and $c$. 

\item  $\Phi^{\text{\rm im}:}=\{n\delta\,|\, 0\not=n\in{\mathbb Z}\}$, the imaginary roots.

\item  $\Phi^{\text{\rm re}}:=\{ j\delta+\alpha,\,|\, j\in{\mathbb Z}, \alpha\in\oPhi\}$, the real roots. 

\item $\Phi^\vee=\{\alpha^\vee\,|\,\alpha\in\Phi^{\text{\rm re}}$, affine roots.

\item $\Phi=\Phi^{\text{\rm re}}\cup\Phi^{\text{\rm im}}$, the root system of $\g$.

\item
$\Phi^+=\{j\delta\,|\, j\in {\mathbb Z}^+\}\cup \{j\delta+\alpha\,|\, j\in{\mathbb Z}^+, \alpha\in\oPhi\}\cup\oPhiplus$:
positive roots.   If $\alpha\in\Phi$, ${\mathfrak g}_\alpha\subset\mathfrak g$ is the $\alpha$-root space. The algebra $\g$ has Borel subalgebra ${\mathfrak b}:=
\langle \fh, \g_\alpha, \alpha\in\Phi^+\rangle$.  Let $\alpha_0=\delta-\theta_l$, so that
$\Pi:=\{\alpha_0,\alpha_1,\cdots,\alpha_r\}$ is the set of simple roots for $\fg$. Then $\alpha_0^\vee=c-\theta_l^\vee$.  Put $\rho=\overset\circ\rho + g\chi$.

\item $Q={\mathbb Z}\alpha_0\oplus\cdots\oplus{\mathbb Z}\alpha_r$, $Q^+={\mathbb N}\alpha_0\oplus\cdots
\oplus{\mathbb N}\alpha_r$: root lattice, positive root lattice.

\item $W:=\langle s_\alpha\,|\,\alpha\in\Phi^{\text{re}}\rangle=\langle s_\alpha\,|\, \alpha\in\Pi\rangle$:
Weyl group of $\g$.

\item For $\lambda\in\fh^*$, $\Phi(\lambda):=\{\alpha\in\Phi^{\text{\rm re}}\,|\, \langle\lambda+\rho,\alpha^\vee\rangle\in{\mathbb Z}\}$.

\item $\sC:=\{\lambda\in\fh^*\,|\, (\lambda+\rho)(c)=\langle \lambda+\rho,c\rangle\not=0\}$, the non-critical region.

\item $\sC^-:=\{\lambda\in\sC\,|\, \langle\lambda+\rho,\alpha^\vee\rangle\leq 0$, for all $\alpha\in\Phi^+(\lambda)=
\Phi(\lambda)\cap\Phi^+\}$, the ``non-critical" anti-dominant chamber.
\end{enumerate}

\section{Module categories}
Following \cite{KT2}, let  ${\mathbb O}={\mathbb O}({\mathfrak g})$
be the category of $\mathfrak g$-modules $M$ which are weight modules for $\fh$ having finite dimensional
weight spaces $M_\lambda$, $\lambda\in{\fh}^*$, and which have the
property that, given $\xi\in{\mathfrak h}^*$, the weight space $M_{\xi+\sigma}\not=0$ for only finitely
many $\sigma\in Q^+$.\footnote{This category, used in \cite{KT2}, is quite close to the original category,
denoted $\sO$
in \cite{Kac}. Indeed, the latter category is contained in $\bO$, and any object $M$ in $\bO$ in which 
$[M:L(\lambda)]\not=0$ implies $\lambda\in\sC$ is a direct sum of objects in the category $\sO$,
as follows from \cite[Prop. 3.3]{KT2}. Following \cite{KT2}, the symbol ``$\sO$" means something different in this paper.}

 Any $\lambda\in {\mathfrak h}^*$ defines a one-dimensional module (still denoted $\lambda$) for the universal enveloping algebra
 $U({\mathfrak b})$ of $\mathfrak b$. Let $M(\lambda):=U({\mathfrak g})\otimes_{U({\mathfrak b})}\lambda$ be the Verma module for $\mathfrak g$ of highest weight $\lambda$. It has
a unique irreducible quotient module $L(\lambda)$. Both $M(\lambda)$ and $L(\lambda)$ belong to
$\bO$.

The category $\mathbb O$ has a contravariant, exact duality $M\mapsto M^\star$, obtained
by taking $M^\star_\lambda$ to be the dual $M^\star_\lambda$ of the $\lambda$-weight space of
$M$,  and making
$\g$ act on $M^\star$ through the Chevalley anti-involution (sending $e_i,f_i,h$ to $e_i,f_i,-h$, respectively). Then $M^{\star\star}\cong M$, 
$M\in\bO$. Also, $M$ and $M^\star$ have the same weight space dimensions (i.~e., they have the
same formal character). Clearly, $L(\lambda)^\star\cong L(\lambda)$. 
In particular, the dual Verma module $M^\star(\lambda):=M(\lambda)^\star$ has socle $L(\lambda)$. 
(See \cite[p. 158]{MP} for similar notation of duality on a slightly different category. See also
\cite[pp. 26--28]{KT2}.)

 For $k\in{\mathbb Q}$, 
let $\bO_k$ be the full subcategory of $\bO$ consisting of all modules
for which the central element $c$ acts as multiplication by the scalar $k$.  For example, $M(\lambda), L(\lambda)\in\bO_k$ for 
$k:=\lambda(c)\in\mathbb Q$. In addition, the duality $M\mapsto M^\star$ restricts to a duality on $\bO_k$.

Put
$$\wg:=[{\mathfrak g},{\mathfrak g}]= ({\mathbb C}[t,t^{-1}]\otimes \go)\oplus {\mathbb C}c,$$
the derived subalgebra of $\mathfrak g$. Let $\wh:=\fh \oplus {\mathbb C}c$, so that
$\wh^*= \ofhstar\oplus {\mathbb C}\chi$. 

Next, define $\Mres(\fg)$ to be the full subcategory of $U(\fg)$-modules $M$ with the property that, given 
any $v\in M$,
$\fg_\alpha v=0$ for all but a finite number of positive roots $\alpha$. The
objects in $\Mres(\fg)$ are {\it not} required to be weight modules; that is, $\Mres(\fg)$ is not a subcategory of $\bO$. Replacing $\fg$ by $\wg$ gives
 a similar category $\Mres(\wg)$
of $\wg$-modules. For $k\in \mathbb Q$, let $\Mres_k(\fg)$ be the full subcategory of $\Mres(\fg)$
consisting of modules upon which $c$ acts by multiplication by $k$.  
A subcategory $\Mres_k(\wg)$ of $\Mres(\wg)$ is
defined
similarly.

The Casimir operator $\Omega$, defined in \cite[\S2.5]{Kac}, belongs
 to a completion $U_c(\g)$ of $U(\g)$  \cite[p. 229]{Kac}.\footnote{\cite{Kac}
only defines $U_c(\wg)$, but the same definition works for $\g$.} In particular,
$\Omega$ defines a locally finite operator on each object in $\Mres(\g)$, commuting with the action
of $\g$. For $a\in\mathbb C$, let $\Mres_{k,a}(\g)$ denote the full subcategory of $\Mres_k(\g)$ consisting of objects $M$ in $\Mres_k(\g)$ upon which $\Omega-a$ acts locally nilpotently, and let
${\mathbb M}^{\text{\rm res},d}_{k,a}(\g)$ be the full subcategory of $\Mres_{k,a}(\g)$ having objects on
which  $d$ has a semisimple action (in addition to the local nilpotence of $\Omega-a$).

\begin{prop} (Kac-Polo)\label{Polo} For $k\not=-g$ and $a\in{\mathbb C}$, there is a full embedding 
$$F_{k,a}:\Mres_k(\wg)\to\Mres_k(\fg)$$ of abelian categories, inducing an equivalence of
$\Mres_k(\wg)\overset\sim\longrightarrow{\mathbb M}^{\text{\rm res},d}_{k,a}(\g)$. Moreover, the inverse of 
the equivalence is given
by restriction.

That is, for each $M\in \Mres_k(\wg)$, $F_{k,a}(M)\in{\mathbb M}^{\text{\rm res},d}_{k,a}(\g)$
and $F_{k,a}(M)|_\wg\cong M$ (naturally). Also, any object in ${\mathbb M}^{\text{\rm res},d}_{k,a}(\g)$
is isomorphic to $F_{k,a}(M)$, for some $M\in\Mres_k(\wg)$.
\end{prop}

\begin{proof} A brief outline of the proof may be found in Soergel \cite[pp. 446-447]{Soe}. We fill
in some details.  First, the algebra
 $U_c(\wg)$ injects  naturally into $U_c(\fg)$, since tensor induction takes
$\Mres(\wg)$ into $\Mres(\fg)$.  If $T_0\in U_c(\wg)$ denotes the ($0$th) Sugawara operator,  the discussion in \cite[p. 228--229]{Kac} shows that 
the equation 
$$T_0=-2(c+g)d +\Omega$$
holds in $U_c(\fg)$.  
 By \cite[Lemma 12.8]{Kac},  the equation
$[T_0,x]=[-2(c+g)d,x]$, for $x\in\wg$, holds in $U_c(\wg)$. If we consider the corresponding
equation of operators on an object $M\in \Mres_k(\wg)$, we may replace $c+g$ by $k+g\not=0$.
 Letting $d$ act as the operator $\tau:=\frac{T_0-a}{-2(k+g)}$
gives an action of $\fg$ on $M$. Equivalently,
$\tau x -x\tau$ acts as $[d,x]$ on $M$ for each $x\in \wg\subseteq U_c(\wg)$. The operator $T_0$ is locally finite on $M$ \cite[p.229]{Kac}, as is $\tau$. For any complex number $\epsilon$, and any positive integer $n$,
let 
$$M_{n,\epsilon}:=\{m\in M\,|\, (\tau-\epsilon)^nm=0\}.$$
 If $x\in {\mathfrak g}$ is a $\gamma$-eigenvector for ad$\,d$, 
we easily find, by induction on $n$, that $xM_{\epsilon,n}\subseteq M_{\gamma+\epsilon,n}$. (Alternatively,
see \cite[Prop. 2.7]{Carter}.) Let
$\tau_s$ be the semisimple part of the locally finite operator $\tau$ on $M$. The operator $\tau_s$
acts as multiplication by $\epsilon$ on $M_{\epsilon,n}$, and by $\gamma+\epsilon$ on $M_{\gamma+\epsilon,n}$. For $m\in M_{\epsilon,n}$ and
$x$ as above, we have
$$\tau_s(xm)=(\gamma+\epsilon)xm=\gamma xm +\epsilon xm =[d,x]m + x\tau_sm$$
and so $\tau_sx-x\tau_s$ acts as $[d,x]$ on  $M$. This, letting $d$ act as $\tau_s$, gives an 
action
of $\wg$ on $M$, extending that of $\wg$. (Note that $\wg$ is spanned by the eigenvectors of $\ad\,d$.)

The constructed $\g$-module belongs to $\Mres_k(\fg)$ and $d$ acts semisimply (as $\tau_s$). The equation
$$\tau= d +\frac{\Omega -a}{-2(k+g)}$$
shows that $\Omega -a$ acts as a nonzero scalar multiple of the locally nilpotent part of $\tau$. So it is
itself locally nilpotent. Finally, the assignment of $M$ to the constructed $\g$-module is clearly
functorial providing a functor $F_{k,a}:\Mres_k(\wg)\to{\mathbb M}^{\text{\rm res}}_{k,a}(\g)$
with $F_{k,a}(M)|_{\wg}=M$.  The construction shows that $F_{k,a}(M)\in{\mathbb M}^{{\text{\rm res}},d}_{k,a}$,
and, clearly, any object $N\in{\mathbb M}^{{\text{\rm res}},d}_{k,a}$ satisfies $N\cong F_{k,a}(N|_\wg)$.
(Note that $d$ must act as the semisimple part of $\frac{T_0-a}{-2(k+g)}$ and $\frac{\Omega-a}{-2(k+g)}$ must
act as the locally nilpotent part.) This completes the proof.
\end{proof}

We  will assume for the rest of this paper, unless otherwise
explicitly stated to the contrary, that $k$ is a rational number with $k+g<0$.
We next define below, for such a $k\in\mathbb Q$, a category $\sO_k$ of $\wg$-modules.  
The definition is taken from \cite{T}, adapted from \cite{KL}. In Corollary \ref{corbelow}, $\sO_k$ is shown to be
equivalent to a category of $\fg$-modules, and is more fully integrated into the $\fg$-module theory in 
\S5. (See Remark \ref{natural order remark}(b).)

Given any $\wg$-module $M$ and
positive integer $n$, let $M(n)$ be the subspace of all $m\in M$ such that $x_1\cdots x_nm=0$
for any choice of $x_1,\cdots, x_n\in\nilrad$. Now define $\sO_k$ to be the full subcategory
of $\wg$-modules $M$ such that (a) $c$ acts as multiplication by $k$; (b) each $M(n)$ is finite
dimensional; and (c) $M=\bigcup_{n\geq 1}M(n)$. Since each $M(n)$ is
evidently a $\og$-submodule of $M|_\og$, (c) implies that $M$ is a locally finite, hence semisimple, $\og$-module.  Then, by (a),  $M$ is a weight module for $\wg$, in the sense that it
decomposes into weight spaces for $\wh$. In addition, $\sO_k$ is a full subcategory of $\Mres_k(\wg)$. 

In \cite[Defn. 2.15, Thm. 3.2]{KL}\footnote{In this reference, the authors define a category $\sO_\kappa$ which
turns out to be $\sO_k$ for $k=\kappa-g$.  The discussion is given only for the simply laced root system
case, but this restriction is not necessary \cite{L}.}, it is shown that all objects in $\sO_k$ have finite length. The
irreducible modules involved are all generated by a highest weight vector having weight $\lambda$
satisfying $\langle\lambda,\alpha_i^\vee\rangle\geq 0$, for $i=1,\cdots, r$ and $\lambda(c)=k$.   These irreducible modules are non-isomorphic for
distinct $\lambda$ above. Conversely, any $\wg$-module with a finite composition series having irreducible quotients of this form belong to
$\sO_k$. 

At the level of $\fg$-modules define $\bO^+=\bO^+(\fg)$ to be the full subcategory of $\bO$ consisting
of all objects $M$ such that $[M:L(\mu)]\not=0$ implies $\langle\mu,\alpha_i^\vee\rangle\in{\mathbb N}$ for
$i=1,\cdots, r$. Let $\bO^{+,\text{finite}}$ be the full category of $\bO^+$ consisting of objects which
have finite length. 

Similarly, for any $k\in\mathbb Q$ with $k<-g$, let $\bO_k=\bO_k(\fg)$ be the full subcategory of $\bO$ consisting of 
all objects upon $c$ acts by multiplication by $k$, and let $\bO^+_k$ be the full subcategory of $\bO$  consisting
of all objects in both $\bO_k$ and $\bO^+$.  If $a\in\mathbb C$, let $\bO^+_{k,a}$ be the full subcategory of $\bO^+_k$ consisting of
all objects upon which $\Omega$ acts with generalized eigenvalue $a$. 
Finally, let $\bO_k^{+,\text{finite}}$ and $\bO_{k,a}^{+,\text{finite}}$ be the full subcategories of $\bO_k^+$
and $\bO^+_{k,a}$, respectively, consisting of objects of finite length.\footnote{Any indecomposable object
of $\bO_k^+$, or of $\bO_k$, or any object in a single block of $\bO_k$, already has finite length. The argument is given below. }

\begin{cor}\label{corbelow} Suppose $k\in\mathbb Q$, $k<-g$. Then the restriction to $\wg$ of any object in $\bO^{+,\text{\rm finite}}_{k,a}$ 
belongs to 
$ \sO_k$. Conversely, if $M\in\sO_k$, then $F_{k,a}M$ belongs to $\bO^{+,\text{\rm finite}}_{k,a}$. These two
functors are mutually inverse, up to a natural isomorphism, and provide an equivalence
$$\bO^{+,\text{\rm finite}}_{k,a}\overset\sim\to \sO_k,$$
\end{cor}

Given any weight $\mu\in\fh^*=\ohstar\oplus{\mathbb C}\chi\oplus{\mathbb C}\delta$, define
$\omu$ to be the projection of $\mu$ into $\ohstar$, and, for $k\in\mathbb Q$, put
$\mu^k=\omu+k\chi\in (\wh)^*$.   Also, $k\not=-g$ (the dual Coxeter number) put 
$$\mu^{k,a}:=\mu^k +b\delta,$$
where $b=\frac{a-(\omu+2\overset\circ\rho,\omu)}{2(k+g)}$ depends on $k$ and $\omu$, as well as $a$.
Note that $\mu|_{\oh}=\omu$ and $\mu|_\wh=\mu^k$ if and only if $\mu$ has level $k$. Also, $\mu=\mu^{k,a}$ if
and only if
$\mu$ has level $k$ and the Casimir operator $\Omega$ acts with eigenvalue $a$ on $L(\mu)$, the
irreducible $\fg$-module of high weight $\mu$. (This is an easy calculation from \cite[Prop. 11.36]{Carter}. See also \cite[p. 229]{Kac}.) Since
we regard $(\wh)^*$ and $\ohstar$ as contained in $\fh^*$, $\mu^k$ and $\mu^{k,a}$ are defined for
$\mu$ in these spaces as well.

As a corollary of this discussion, we have the following.
\begin{prop}\label{restrictionweights} Suppose $\mu,\mu'\in\sC$ and $\mu'=w\cdot\mu$ for some
element $w\in W$. (Here $w\cdot\mu:=w(\mu+\rho)-\rho$ is the usual dot action of $W$.) Then,
if $\mu\not=\mu'$, we have $\mu|_{\wh}\not= \mu'|_{\wh}$. 

\end{prop}

\begin{proof}The
levels of $\mu,\mu'$ have the value $k=(\mu+\rho,\delta)=(w(\mu+\rho),\delta)= (\mu'+\rho,\delta)$, after
noting that $(\rho,\delta)=0$. By \cite[Prop. 11.36]{Carter},
$\Omega$ acts on $L(\mu)$ and $L(\mu')$ by multiplication by
$$\begin{aligned} a &=(\mu+\rho,\mu+\rho)-(\rho,\rho)\\
&=(w(\mu+\rho),w(\mu+\rho))-(\rho,\rho)\\ &=(\mu'+\rho,\mu'+\rho)-(\rho,\rho).\end{aligned}$$
If $\mu|_{\wh}=\mu'|_{\wh}$, then $\overset\circ\mu=\overset\circ\mu'$ and (by above) 
$$\mu=\mu^{k,a}=\overset\circ\mu^{k,a}=\overset\circ\mu^{\prime,k,a}=\mu',$$
as required.\end{proof}

 \section{Weyl groups and linkage classes} Maintain the above notation. For $\alpha\in\Phi^{\text{\rm re}}$,  form the reflection $s_\alpha:\fh^*\to \fh^*$,
 $x\mapsto x-\langle x,\alpha^\vee\rangle\alpha$. The Weyl group   
 $$W=\langle s_\alpha\,|\,\alpha\in\Phi^{\text{\rm re}}\rangle$$
 for $\g$ is a Coxeter group with fundamental reflections
  $S=\{s_{\alpha_i}\,|\,i=0,\cdots, r\}$.   For $\lambda\in \fh^*$, 
  $$\Phi(\lambda):=\{\alpha\in\Phi^{\text{\rm re}}\,|\,
 \langle\lambda+\rho,\alpha^\vee_i\rangle\in{\mathbb Z}\}$$
 is a subroot system of $\Phi$, in the sense of \cite{KT2}.  The subgroup 
 $$W(\lambda)=\langle s_\alpha\,|\,\alpha\in\Phi(\lambda)\rangle$$ 
 is a Coxeter group with fundamental system consisting of
 the $s_\alpha$, $\alpha\in\Phi^+(\lambda)=\Phi^+\cap\Phi(\lambda)$, such that $s_\alpha(\Phi^+(\lambda)\backslash\{\alpha\})=
 \Phi^+(\lambda)\backslash\{\alpha\}$. Then $W_0(\lambda)=\{w\in W(\lambda)\,|\, w\cdot\lambda=\lambda\}$
 is also a Coxeter system, generated by reflections  $s_\alpha$ with $\Phi_0(\lambda)=\{\alpha\in\Phi^{\text{\rm re}+}\,|\,\langle\lambda+\rho,\alpha^\vee\rangle=0$. Let $[\lambda]:=W(\lambda)\cdot\lambda$ and $[\lambda]^+=\{\mu\in[\lambda]\,|\,(\mu,\alpha_i^\vee)\in{\mathbb N}, 1\leq i\leq r\}$. An element $\lambda\in\sC^-$ is
 called regular if $W_0(\lambda)=\{1\}$.
  
 Let 
 $\bO[\lambda]$
   be the full subcategory of $\bO$ 
 consisting of modules $M$ such that $[M:L(\mu)]\not=0$ implies $\mu\in[\lambda]$. 
 
 Let $\lambda\in\sC^-$. Then $\lambda$ is the smallest element in $[\lambda]$, in the sense that
 all weights in $[\mu]$ belong to $W(\lambda)\cdot\lambda\subseteq\lambda+Q^+$. In particular, all objects of $\bO[\lambda]$ have finite length.
 As shown in \cite[Prop. 3.1]{KT2}, using work of Kac-Kazhdan \cite{KK}, if $\mu\in[\lambda]$, the Verma
 module $M(\mu)$ belongs to $\bO[\lambda]$.  Also,  $M^*(\mu)$ also belongs to
 $\bO[\lambda]$.  If $M\in\bO$ is indecomposable with $[M:L(\mu)]\not=0$, then $M\in\bO[\lambda]$
 \cite[Prop. 3.2]{KT2}. Or if $M$ is an indecomposable module linked to $L(\mu)$ via a chain of 
 indecomposable
 modules $M_i$ in $\bO$, $i=0, \cdots, n$ and irreducible modules $L(\mu_i)$ satisfying 
 $M=M_0$ and $ [M_i:L(\mu_i)]\not=0\not=[M_{i+1}:L(\mu_i)]$, and $\mu-\mu_n$, then $M\in\bO[\lambda]$.
 Moreover, if $M=L(\nu)$ with $\nu\in[\lambda]$, there is such a chain, using \cite[Prop. 31]{KT2} (which
 quotes \cite{KK}). That is, $\bO[\lambda]$ is the ``block" of $\bO$ associated to $L(\mu)$, and the
 irreducible modules $L(\nu)$, $\nu\in[\lambda]$, constitute a linkage class.

The inclusion $i_*:\bO^+\to\bO$ admits a right adjoint $i^!$ and a left adjoint $i^*$. Explicitly, given
$M\in\bO$,  $M_+=i^!M$ (resp., $M^+=i^*M)$ is the largest submodule (resp., quotient module) lying
in $\bO^+$.  For $\lambda\in\sC^-$, $\bO^+[\lambda]$ will denote the full subcategory of $\bO[\lambda]$ consisting of all $M$ such that $[M:L(\mu)]\not=0$ implies $\mu\in[\lambda]^+$.

 Consider the parabolic subalgebra 
 $${\mathfrak p}:= ({\mathbb C}[t]\otimes \og)\oplus {\mathbb C}c\oplus{\mathbb C}d$$
 of $\g$. Its Levi factor is denoted
 $${\sL}:=\og\oplus {\mathbb C}c\oplus{\mathbb C}d=\og +\fh.$$
 Any $\lambda\in \fh^*$ determined an irreducible ${\sL}$-module ${\mathfrak s}(\lambda)$
 whose restriction to $\og$ is always irreducible. Writing $\lambda=\overset\circ\lambda + k\chi + b\delta$,
 $k,b\in\mathbb C$,
 ${\mathfrak s}(\lambda)$ is finite dimensional if and only if $\overset\circ\lambda\in X^+$ (and in this case its
 restriction to $\og$ is then the finite dimensional irreducible module $V(\overset\circ\lambda)$ of highest
 weight $\overset\circ\lambda$).
 
 \begin{prop}Let $\mu=\overset\circ\mu + k\chi + b\delta\in\fh$, with $k<-g$,  and $\overset\circ\mu
 $  dominant on $\overset\circ\fh$.   Then
 $$\begin{cases}M(\mu)^+\cong U(\fg)\otimes_{U(\frak p)}{\mathfrak s}(\mu);\\
  M(\mu)_+\cong\Hom_{U({\mathfrak p})}(U(\fg),{\mathfrak s}(\mu)).\end{cases}$$
  Also, $(M(\mu)^+)^\star\cong M^*(\mu)_+$.\end{prop}
 
 \begin{proof} Each subspace $t^n\otimes\og\subset\fg\subset U(\g)$, $n\in\mathbb Z$, is a $\og$-submodule
 of $U(\g)$ under the adjoint action, isomorphic to the adjoint module $\og$. Obviously, $U(\g)\otimes_{U({\mathfrak p})}{\mathfrak s}(\mu)$,  as a left $\g$-module under multiplication, is the homomorphic image
 of a direct sum of modules $(t^{n_1}\otimes\og)\otimes(t^{n_2}\otimes\og)\otimes\cdots\otimes (t^{n_m}\otimes
 \og)\otimes{\mathfrak s}(\mu)$, $n_1,\cdots, n_m\in{\mathbb Z}^{\leq 0}$, $m\in {\mathbb N}$.  All these tensor products are finite dimensional $\og$-modules. Thus, if $L(\nu)$ is a $\g$-composition factor of
 $U(\g)\otimes_{U({\mathfrak p})}{\mathfrak s}(\mu)$, its highest weight space must generate a finite
 dimensional $\og$-module. In particular, $(\nu,\alpha_i^\vee)\in{\mathbb N}$, for $i=1,\cdots n$.
 
Let $U(\sL)\otimes_{U(\sL\cap{\mathfrak b})}\mu$ be the Verma module for $\sL$ with  
 highest weight $\mu$. It inflates naturally to $\mathfrak p$ as $U({\mathfrak p})\otimes_{U({\mathfrak b})}\mu$. There is an exact sequence 
 \begin{equation}\label{exact}0\to N\to U({\mathfrak p})\otimes_{U({\mathfrak b})}\mu\to {\mathfrak s}(\mu)\to 0
 \end{equation}
 of $\mathfrak p$-modules. From the classical theory of $\og$-Verma modules, the highest weight $\varpi$
 of any composition factor $Y$ of $N$ has the property that $(\varpi,\alpha_i^\vee)<0$ for some $i$. The
 section
 $U(\g)\otimes_{U({\mathfrak p})}Y$ of $U(\g)\otimes_{U({\mathfrak p})}N$ has an irreducible head
 with the same highest weight $\varpi$. 
  
  It follows now, by tensor inducing  the exact sequence (\ref{exact}) of $\mathfrak p$-modules, that
  $U(\g)\otimes_{U({\mathfrak p})}{\mathfrak s}(\mu)$ is the largest  quotient of $M(\mu)=
  U(\g)\otimes_{U({\mathfrak b})}\mu$ with all composition factors in $\bO^+$.  That, $M(\mu)^+
  \cong U(\g)\otimes_{U({\mathfrak p})}{\mathfrak s}(\mu)$.
  
  A similar argument establishes the assertion for $M(\mu)_+$ and the final assertion is obvious. \end{proof}

 \section{Highest weight categories}Throughout this section, fix $\lambda\in\sC^-$. The set $[\lambda]=W(\lambda)\cdot\lambda$ is a poset, putting $\mu\leq\nu$ if and only if   
  $\nu-\mu\in Q^+:=\sum_{i=0}^r{\mathbb N}\alpha_i$.\footnote{If $\Gamma$ is a poset ideal in a poset
  $\Lambda$, i.~e., if $\nu\leq\gamma\in\Gamma\implies\nu\in\Gamma$, we write $\Gamma\trianglelefteq\Lambda$. We will also consider other partial orders on $[\lambda]$ in this section.} The set $[\lambda]$ has a unique minimal element,
  namely, $\lambda$. If $\nu\in[\lambda]$, then $\{\mu\in[\lambda]\,|\,\mu\leq\nu\}$ is a finite poset ideal
  in $[\lambda]$. If $\Gamma\subseteq[\lambda]$, let $\bO[\Gamma]$ be the full subcategory of $\bO$
  consisting of objects with have ``composition factors" (in the sense of \cite[p.151 ]{Kac}) $L(\gamma)$, $\gamma\in\Gamma$. It will often convenient (in \S\S6,7) to denote $\bO[\Gamma]$ by $\bO^\Gamma[\lambda]$ where $\lambda$ needs to be mentioned.

We will consider the following categories of $\g$-modules. In each case, the irreducible modules are
indexed, up to isomorphism, by a poset $\Gamma$ in $[\lambda]$ or $[\lambda]^+$.   In case $\Gamma\subseteq[\lambda]^+$, we let $\bO^+[\Gamma]= \bO^{\Gamma,+}[\lambda]$ be the full subcategory of $\bO^+$ consisting of
modules composition factors $L(\nu)$, $\nu\in\Gamma$. Also, for $\nu\in\Gamma$, there are given two modules
$\Delta(\nu)$ and $\nabla(\nu)$ in the category.

An object $X$ in $\bO^\infty[\lambda]$ (resp., $\bO^{+,\infty}[\lambda]$) is by definition a directed union $\{X^\Gamma\}_\Gamma$ of modules
$X^\Gamma\in\bO^\Gamma[\lambda]$ ranging over finite $\Gamma\trianglelefteq [\lambda]$ (resp., $\Gamma^+\trianglelefteq [\lambda]^+$). 
 
 \begin{thm}\label{First theorem of section 5}  For $\lambda\in\sC^-$, each of the categories listed in Table \ref{variouscategories}  below is a highest weight category (in the sense of \cite{CPS-1})
 with standard (Weyl) modules $\Delta(\nu)$, costandard modules $\nabla(\nu)$ and indicated poset.
 In particular, each of these categories has enough injective objects. \end{thm}
 
 \begin{proof} The fact that $\bO[\Gamma]$ is a highest weight category follows from the dual
of the definition \cite{CPS-1} and the fact that projectives in suitable ``truncated" categories have
 Verma module filtrations \cite[Lemma 10]{RR}. The latter reference does not use an arbitrary
 poset ideal $\Gamma$, but every such $\Gamma$ is contained in one of theirs, which is sufficient, see  
  \cite[Lemma 2.3]{Fiebig0} for the case when $\Gamma$ is any ideal generated by a single element.
  
In \cite{RR}, the authors also treat the parabolic case, and generalized Verma modules. Thus, 
it follows similarly that $\bO^{+}[\Gamma]$ is a highest weight category. 

If $\Gamma\leq\Gamma'$ are two poset ideals, then the injective hull $I_\Gamma(\nu)$ of any
given irreducible  $L(\nu)$ in $\bO[\Gamma]$ embeds in the injective hull $I_{\Gamma'}(\nu)$ in 
$\bO[\Gamma']$. Taking a directed union over the chain of poset ideals $\Gamma_n$, $n\in{\mathbb N}$,
with $\bigcup \Gamma_n\supseteq[\lambda]$ gives an injective hull in $\bO^\infty[\lambda]$. It follows easily
from \cite{CPS-1} that $\bO^\infty[\lambda]$ is a highest weight category. Similarly, $\bO^{+,\infty}[\lambda]$ is
also a highest weight category.   
 \end{proof}
 
\begin{table}
\begin{center}
\begin{tabular}{| l | r | r | r |}
\hline
Category & poset & $\Delta(\nu)$ & $\nabla(\nu)$ \\  \hline
$\bO[\Gamma] $ & $\Gamma\trianglelefteq[\lambda]$ & $M(\nu)$ & $M^\star(\nu)$\\  \hline
$\bO^{+}[\Gamma]$ & $\Gamma\trianglelefteq [\lambda]^+$ & $M(\nu)^+$& $M^\star(\nu)_+$\\ \hline
$\bO^\infty[\lambda]$ & $[\lambda]$ & $M(\nu)$ & $M^\star(\nu)$\\ \hline
$\bO^{+,\infty}[\lambda]$ & $[\lambda]^+$ & $M(\nu)^{+}$ & $M^*(\nu)_{+}$\\ \hline
\end{tabular}
\caption{Various categories}\label{variouscategories}
\end{center}
\end{table} 
 
 \begin{rem}\label{categories} (a)  
  The category $\bO[\lambda]$, $\lambda\in\sC^-$, is the union $\bigcup_\Gamma\bO[\Gamma]$ over all finite poset ideals $\Gamma\subset[\lambda]$ (with respect to $\leq$). The category $\bO[\lambda]$ satisfies
 most of the axioms in \cite{CPS-1} for a highest weight category, though not all. (There are not enough
 injective objects.)  However, each full subcategory $\bO[\Gamma]$ does have enough injective
 and projective objects, and is a highest weight category.  As we have seen, the categories $\bO[\Gamma]$ can be used to used to formally complete $\bO[\lambda]$ to a highest weight category
 $\bO^\infty[\lambda]$.  Some authors, speaking more informally, simply call such categories
 (like $\bO[\lambda]$) a highest weight
 category. 
 
 Similar remarks apply to $\bO^+[\lambda]$, defined to be $\bO^+\cap\bO[\lambda]$, as well as many of the
 other categories we have introduced.

(b)  For $\lambda\in \sC^-$, there is another poset structure $\leq^W=\leq^{W(\lambda)}$ on $[\lambda]=W(\lambda)\cdot\lambda$.
 Explicitly, given $\mu\in W(\lambda)\cdot\lambda$, write $\mu=w_\mu\cdot\lambda$ where $w_\mu\in W(\lambda)$ is the unique element $w\in W(\lambda)$ of minimal length such that $w\cdot\lambda=\mu$.
 Then, for $\mu,\nu\in[\lambda]$, put $\mu\leq^W\nu$ if and only if $w_\mu\leq w_\nu$ in the Bruhat-Chevalley
 order on the Coxeter group $W(\lambda)$. We will show in Remark \ref{natural order remark}(a) that the above categories are
 highest weight categories  using $\leq^W$ (or its restriction to $[\lambda]^+$) with
 the same standard and costandard objects \end{rem}

\begin{defn}\label{partialordering}For $\mu,\nu\in\wh^*$, put $\nu\wleq\mu$ provided that the following two conditions hold: 
\begin{enumerate}
\item $\mu,\nu$ have the same level $k\not=-g$, i.~e., $\mu(c)=\nu(c)\not=-g$; 

\item  $\nu^{k,a}\leq \mu^{k,a}$ for some $a\in \mathbb C$. 
\end{enumerate}
\end{defn}

Condition (2) does not actually depend on $a$. Indeed, writing $\mu=\overset\circ\mu+k\chi$, 
$\nu=\overset\circ\nu+k\chi$, then $\mu^{k,a}=\overset\circ\mu+k\chi +b\delta$ and $\nu^{k,a}=\overset\circ\nu+k\chi+ b'\delta$, then 
$$\begin{cases} b=\frac{a-(\overset\circ\mu + 2\orho,\omu)}{k+g}\\
b'=\frac{a-(\overset\circ\nu+2\orho,\overset\circ\nu)}{k+g}.\end{cases}$$
Thus, $\nu^{k,a}\leq\mu^{k,a}$ if and only if
\begin{enumerate}
\item[(1)] $b-b'\in\mathbb N$;
\item[(2)] $\overset\circ\mu-\overset\circ\nu + (b-b')\theta_l\in \overset\circ Q^+$
\end{enumerate}
The parameter $a$ drops out of $b-b'$ which depends only on $\mu,\nu$ and $k$. Thus, any $a$
may be used in defining $\nu\wleq\mu$

For $\lambda\in\sC^-$,  we temporarily write $[\lambda]^\sim$ for the collection of all $\widetilde\mu:=
\mu|\wh$, for $\mu\in[\lambda]$. As a consequence of the discussion, we have:

\begin{prop}\label{aboveprop} The restriction map $[\lambda]\to[\lambda]^\sim$ is a poset isomorphism,
if $[\lambda]$ is given its usual poset structure via $\leq$ above, and if $[\lambda]^\sim$
is given its poset structure via $\wleq$.\end{prop}

\begin{proof} The bijectivity of restriction has already been established in Proposition \ref{restrictionweights}. As noted in
its proof, an inverse on $[\lambda]^\sim$ of restriction is provided by $\widetilde\mu\mapsto\widetilde\mu^{k,a}= \mu^{k,a}$, where $\lambda=\lambda^{k,a}$. As the definition of $\wleq$ shows, this inverse is 
order preserving, as is the restriction map itself.\end{proof}

We introduce some further categories, obtained by restricting to $\wg$ all the categories listed in
Table 1, as well as $\bO[\lambda]$ and $\bO^+[\lambda]$,  decorating the resulting strict image category with a ``tilde" (i.~e., changing $\bO$ to $\widetilde\bO$).
Each of these $\wg$-categories has an associated poset,
given in Table 1 for the corresponding $\g$-category. (We can view the posets as abstract
sets, useful for labeling irreducible, standard, and costandard modules. As such, there is no need to pass to
version using $([\lambda]^\sim,\wleq)$ in view of the poset isomorphism in Proposition \ref{aboveprop}.
 Keeping the Table 1 version eases our notational burden.) Thus, in the proposition below, we use the $\leq$ partial order on $[\lambda]$ and $[\lambda]^+$, whether or not we are dealing with categories of $\fg$ 
 or $\wg$-modules. We will extend the proposition to the partial ordering $\leq^W$ in Remark \ref{natural order
 remark}(a), as well as to additional partial orders $\leq_\nat$ discussed there.
   
 For use below, define 
$$\sC^-_{\text{\rm rat}}:=\{\lambda\in\sC^-\,|\,\langle\lambda,\alpha_i^\vee\rangle\in{\mathbb Z},\,{\text{\rm
for each}}\,
\,i=1,\cdots, r\, {\text{\rm and}}\,\lambda(c)\in{\mathbb Q}\}.$$
  In particular, each $\lambda\in\sC_{\text{\rm rat}}$ has level a rational number $k$ and $k<-g$ since
  $\lambda\in\sC^-$.   Note that, for $\lambda\in\sC^-$, $[\lambda]^+=\emptyset$ unless $\lambda\in \sC^-_{\text{\rm rat}}$.    
  
\begin{prop} \label{ideals}
Let $\lambda\in\sC^-_{\text{\rm rat}}$ and let $\Gamma\trianglelefteq[\lambda]$ (resp.,
$\Gamma^+\trianglelefteq[\lambda]^+$) be finite. The categories $\widetilde\bO[\Gamma]$, $\widetilde\bO^{+}[\Gamma^+]$,
$\widetilde\bO^\infty[\lambda]$, and $\widetilde\bO^{+,\infty}[\lambda]^+$ are all highest weight categories with
weight posets $\Gamma$, $\Gamma^+$, $[\lambda]$ and $[\lambda]^+$, respectively. Each of these categories is equivalent to its counterpart for $\g$-modules, as is each of the categories
$\widetilde\bO[\lambda]$ and $\widetilde\bO^+[\lambda]$. The functors providing these
equivalences are, in each case, given by the restriction functor of $\g$-modules to $\wg$-modules, and by applying
$F_{k,a}$ to objects in, say, a category $\widetilde\bO^\Gamma[\lambda]$, with $a$ determined by $\lambda$.
There is a natural common extension of the functors $F_{k,a}$ to the functors $F^\infty_{k,a}$ on
$\wO^\infty[\lambda]$, which is also inverse to the restriction functor.
\end{prop}

\begin{proof} Let $\bM_k^{{\text{\rm res}},\infty}(\wg), \bM_k^{{\text{\rm res}},d,\infty}(\fg)$ denote the categories of $\wg,\fg$
modules, respectively, which are directed unions of objects in  $\bM_k^{{\text{\rm res}}}(\wg),
 \bM_k^{{\text{\rm res}},d}(\fg)$,
respectively. Then  the functors $F_{k,a}$ extend in an obvious way to functors $F_{k,a}^\infty$ on
the direct union categories, giving equivalences inverse to restriction. In particular, $F_{k,a}^\infty$
provides, $\widetilde\bO^\infty[\lambda]$ and $\widetilde\bO^{+,\infty}[\lambda]$ equivalences inverse to restriction to the versions without the ``tilde." The remaining equivalences are obvious.
\end{proof}

Each of the $\widetilde{\mathfrak g}$-modules categories above has irreducible, standard and costandard modules. These modules will be denoted by placing a ``tilde" over their $\g$-counterparts. Thus, $\wL(\mu)$, $\wM(\mu)$ and $\wM_\star(\mu)$ are
the irreducible, standard and costandard modules ofr $\widetilde\bO[\lambda]$, if $\mu\in[\lambda]$. If
$\mu\in[\lambda]^+$, then $\widetilde M^+(\mu)$, $\widetilde M_+(\mu)$ are the standard and costandard modules
for $\widetilde\bO^+[\lambda]^+$. All of these modules are the restrictions to $\wg$ of their $\g$-counterparts.

\begin{rem}\label{natural order remark} (a) 
 Let $\sE$ be an abstract highest weight category \cite{CPS-1} with weight poset $(\Lambda,\leq)$, and
 with costandard modules $\nabla(\lambda)$, $\lambda\in\Lambda$. Assume there also exist standard
 modules $\Delta(\lambda)$, $\lambda\in\Lambda$, and assume that $\Delta(\lambda)$ and $\nabla(\lambda)$
 have the same composition factors (with multiplicities). 
 There is a natural partial order $\leq_{\nat}$, at least
 when $\Delta(\lambda)$ and $\nabla(\lambda)$ have the same composition factors (which holds
 for all the categories above). More precisely, $\leq_{\nat}$ is the partial order generated by the requirement that $\mu\leq_{\nat} \nu$ when
 $[\Delta(\nu):L(\mu)]\not=0$. Let $\Lambda_\nat=(\Lambda,\leq_\nat)$ denote this new poset. Then $\sE$ is also a
 highest weight category with respect to $\Lambda_\nat$ with the same standard and costandard modules.
 If $\nu\leq_\nat\mu$, then clearly $\nu\leq\mu$. 
 F
 Now suppose that $\preceq$ is a partial order on the set $\Lambda$ such that, for each $\mu,\nu\in\Lambda$, 
 $\mu\leq_\nat\implies\mu\preceq\nu$.  Then it is easily seen that $\sE$ is a highest weight category with
 respect to $(\Lambda,\preceq)$, with the same standard and costandard modules. 
 Moreover, if $\Gamma$ is a $\preceq$-ideal in $\Lambda$, 
 it is also a $\leq_\nat$-ideal in $\Lambda$.  
 
Returning to the situation of Proposition \ref{ideals}, we have,
 in addition to the partial orders $\leq_\nat$ and $\leq$ on $[\lambda]$, there is also the partial order $\leq^W$ discussed in Remark 5.3(b) and a further partial ordering $\uparrow$ (discussed in \S9, Appendix I). 
 The orders $\leq^W$ and $\uparrow$ are shown to the same on $[\lambda]$ in Proposition \ref{First prop of appendix} below.  It is also true
 that $\mu\leq_\nat\nu$ implies in an obvious way, using the remarks above Proposition \ref{First prop of appendix} that $\mu\uparrow\nu$, and, thus, now using Proposition
 \ref{First prop of appendix}, $\mu\leq^W\nu$.
 In turn, $\mu\leq^W\nu$ implies $\mu\leq \nu$, when $\lambda\in\sC^-$. We summarize this discussion
 as follows.
 \begin{equation}\label{implications}
 \mu\leq_\nat\nu\implies\mu\leq^W\nu\implies\mu\leq \nu \,\,(\mu,\nu\in[\lambda]).\end{equation}
 There is a version of these implications for $[\lambda]^+$, when $\lambda\in\sC^-_\rat$.
 \begin{equation}\label{implicationsII}
 \mu\leq_\nat\nu\implies\mu\leq^W\nu\implies\mu\leq \nu\,\,(\mu,\nu\in[\lambda]^+).\end{equation}
 The meaning of $\leq_\nat$ in (\ref{implicationsII}) is not quite the same as it is in (\ref{implications})
 since the $\leq_\nat$ for $[\lambda]^+$ is computed with respect to different standard modules. However,
 if $\mu\leq_\nat\nu$ in (\ref{implicationsII}), then $\mu\leq_\nat\nu$ in (\ref{implications}).  This implies
 the validity of (\ref{implicationsII}). We will mostly be using (\ref{implicationsII}), so we have preferred 
 not to use separate notations for the two orders denoted $\leq_\nat$. (Finally, it is interesting to observe that
 in (\ref{implications}) we have $\leq_\nat=\leq^W$, though we will not need this fact.)
 
 The main conclusion to be drawn is that Proposition \ref{ideals} holds as written, if the order
 $\leq$ is replaced by $\leq^W$ or by $\leq_\nat$, although it must be understood that the meaning
 of $\leq_\nat$ varies between $[\lambda]$ and $[\lambda]^+$. A similar observation holds regarding
 Theorem \ref{First theorem of section 5}. 
 
(b) We can also define categories $\widetilde\bO_k$, $\widetilde\bO^+_k$ and $\widetilde\bO^{+,\text{\rm finite}}_k$, just as the 
 (strict) images under the restriction functor of the corresponding categories $\bO_k$,  $\bO^+_k$ and
 $\bO^{+,\text{\rm finite}}$ of $\fg$-modules. Each such strict image is a full subcategory of $\wg$-modules, by
 Proposition \ref{Polo}. (All of their objects belong to $\bM^{\text{res}}_k(\wg)$.) Also, $\widetilde\bO^{+,{\text{\rm finite}}}_k
 \cong\sO_k$. But it is not true that any of these new categories is equivalent to the category of $\g$-modules
 from which its name is derived, since, in particular, no generalized eigenvector of $\Omega$ has been specified.
 Thus, while $\sO_k\cong\widetilde\bO_k^{+,\text{finite}}$, the category $\sO_k$ is not equivalent to $\bO_k^{+,\text{finite}}$.
 Instead, $\sO_k$ is equivalent to $\bO^{+,\text{finite}}_{k,a}$ as Corollary \ref{corbelow} shows. (It is true that $\bO^{+,\text{finite}}_k$ is equivalent to a direct sum of copies of $\sO_k$.) This is not an issue, however, with the $\wg$-categories in Theorem \ref{First theorem of section 5}, since for example the generalized eigenvalue of the
 Casimir operator $\Omega$ on objects of $\bO^\infty[\lambda]$ is determined by $\lambda$. 
\end{rem}

In \cite{T}, Tanisaki described the group $W(\lambda)$ explicitly for any $\lambda\in \sC^-_{\text{\rm rat}}$.  He also describes the dot action of $W(\lambda)$ on
$[\lambda]$ in more explicit terms. Implicit in his discussion is a description of $\Phi(\lambda)$, together
with a set of fundamental roots (which can in any event be easily calculated). We return to this in \S7.1.
 
\section{Translation functors}
Let $P$ be a fixed $\mathbb Z$-lattice ${\mathfrak h}^*$ whose projection onto $\wh^*$ relative
to the decomposition (\ref{decomp}) is $\overset\circ P +{\mathbb Z}\chi$, where $\overset\circ P$
is the weight lattice of $\oPhi$ in $\ofhstar$.  Let $P^+$ be the set of all $\lambda\in P$ such that
$\langle \lambda,\alpha^\vee_i\rangle\in{\mathbb N}$ for all $i=0,\cdots, r$.

Given $\lambda,\mu\in\sC^-$ with $\mu-\lambda\in WP^+$, Kashiwara-Tanisaki \cite{KT2} define an
exact
translation functor
$$T^\lambda_\mu:\bO[\lambda]\to\bO[\mu].$$
The definition is the familiar one, taking a ``block" projection, after tensoring with an irreducible module
$L(\gamma)$, $\gamma\in P^+\cap W(\mu-\lambda)$. We summarize the key properties they prove
in the following proposition.

\begin{prop} \label{KTprop}(\cite[Prop. 3.6, 3.8]{KT2}) Assume $\lambda,\mu\in \sC^-$ satisfy $\mu-\lambda\in WP^+$
and $\Phi_0(\lambda)\subseteq\Phi_0(\mu)$. 

(a) For any $w\in W(\lambda)$,
$$T^\lambda_\mu M(w\cdot\lambda)\cong M(w\cdot\mu).$$

(b) Also,
 $$T^\lambda_\mu L(w\cdot\lambda))=\begin{cases} L(w\cdot\mu),\quad{ \text{\rm if}}\, w(\Phi^+_0(\mu)\backslash
   \Phi_0^+(\lambda))\subseteq\Phi^+(\lambda);\\
   0,\quad {\text{\rm otherwise.}}\end{cases}$$
   \end{prop}

\begin{rem}\label{minimal length} Recall that if $\lambda'\in[\lambda]$ we let $w_{\lambda'}\in W(\lambda)$
be the shortest element $x\in W(\lambda)$ such that $w\cdot\lambda=\lambda'$. The condition
$w(\Phi^+_0(\mu)\backslash\Phi^+_0(\lambda))\subseteq\Phi^+(\lambda)$ in Proposition \ref{KTprop}(b) is
equivalent to $w_{\lambda'}=w_{\mu'}$ where $\lambda'=w\cdot \lambda$ and $\mu'=w\cdot\mu$. 
In case $\lambda$ is regular, the condition is that $w=w_{\mu'}$. 
 \end{rem}

We add some additional useful properties of $T^\lambda_\mu$ after extending these functors in two ways.
First, the  ``block" projection definition of translation extends easily to $\bO^\infty[\lambda]$.  We use the
same notation $T^\lambda_\mu$ for this extension, so that now we have the functor
$$T^\lambda_\mu:\bO^\infty[\lambda]\to\bO^\infty[\mu].$$
Second, $T^\lambda_\mu$ and the natural equivalences $\widetilde\bO[\lambda]\cong
\bO^\infty[\lambda]$ and $\bO^\infty[\mu]\cong\widetilde\bO^\infty[\mu]$ define a composite
functor
$$\wT^\lambda_\mu:\widetilde\bO^\infty[\lambda]\to\widetilde\bO^\infty[\mu].$$
Therefore, there is a commutative diagram
$$\begin{matrix} \bO^\infty[\lambda]&\overset{T^\lambda_\mu}\longrightarrow & \bO^\infty[\mu] \\
\downarrow\sim &&\downarrow\sim\\
\widetilde\bO^\infty[\lambda]&\overset{\widetilde T^\lambda_\mu}\longrightarrow & \widetilde\bO^\infty[\mu]
\end{matrix}
$$
where the vertical maps are just restriction of functors. 

 We now list these additional properties of $T^\lambda_\mu$.
\begin{prop}\label{prop6.3}Assume $\lambda,\mu\in \sC^-$ satisfy the properties $\Phi_0(\lambda)\subseteq\Phi_0(\mu)$
and $\mu-\lambda\in WP^+$ of Proposition \ref{KTprop}. Then
the following statements hold.

(a) $T^\lambda_\mu M^*(w\cdot\lambda)\cong M^*(w\cdot\mu)$ and $\wT^\lambda_\mu\widetilde M^*(w\cdot
\lambda)\cong \widetilde M^*(w\cdot\mu)$.

(b) If $\Phi_0(\lambda)=\Phi_0(\mu)$, then $T^\lambda_\mu$ and $\wT^\lambda_\mu$ 
give equivalences of categories 
$$\begin{cases}T^\lambda_\mu:\bO^{\infty}[\lambda]\overset\sim\to \bO^{\infty}[\mu]\\
\widetilde T^\lambda_\mu:\widetilde\bO^{\infty}[\lambda]\overset\sim\to\widetilde\bO^{\infty}[\mu].\end{cases}$$

(c) Again assume that $\Phi_0(\lambda)=\Phi_0(\mu)$. If $\Gamma\subseteq[\lambda]$, 
set $\Gamma'=\{w\cdot\mu\,|\,w\cdot\lambda\in\Gamma\}$.
Then $\Gamma$ is a poset ideal in $([\lambda],\leq^W)$ if and only if $\Gamma'$ is a poset ideal in 
$([\mu],\leq^W)$. In this case,  the posets $\Gamma$ and  $\Gamma'$ are isomorphic by the evident
map $w\cdot\lambda\mapsto w\cdot\mu$, and the functors $T^\lambda_\mu$
and $\wT^{\lambda}_\mu$ induce (by restriction) category equivalences 
$$\begin{cases}T^\lambda_\mu:\bO^{\Gamma}[\lambda]\overset\sim\to \bO^{\Gamma'}[\mu]\\
\wT^\lambda_\mu:\widetilde\bO^{\Gamma}[\lambda]\overset\sim\to\widetilde\bO^{\Gamma'}[\mu].
\end{cases}$$  
 \end{prop}
 \begin{proof} We first prove (a). Because $L(\mu)^*\cong L(\mu)$, the duality on $\bO[\mu]$ preserves
 blocks in the category $\bO[\lambda]$. It follows thus follows by construction that the translation
 functor $T^\lambda_\mu$ commutes with duality. Now apply Theorem \ref{KTprop}(a).
 
 Next, we consider (c). First,  the assertions concerning $\Gamma$ and $\Gamma'$ follow from Remark \ref{minimal length}.
 Because the exact functor $T^\lambda_\mu$ takes standard (resp., costandard) modules $\Delta(w\cdot
 \lambda)$ (resp. $\nabla(w\cdot\lambda)$), $w\cdot\lambda\in \Gamma$, in
 $\bO[\lambda]$ 
 to standard (resp., costandard) modules $\Delta(w\cdot\mu)$  (resp., $\nabla(w\cdot\mu)$) in $\bO[\mu]$, the comparison theorem \cite[Thm. 5.8]{PS-1} implies it is an equivalence of categories. A similar argument
 applies to $\wT^\lambda_\mu$. This proves (c). 
 
 Finally, (b) follows from (c), writing $[\lambda]$ as a (directed) union of finite ideals $\Gamma$.
  \end{proof}

\begin{lem}\label{important lemma} Let $\lambda,\mu\in\sC^-_\rat$ satisfying the conditions $\mu-\lambda\in WP^+$ and $\Phi_0(\lambda)\subseteq\Phi_0(\mu)$. Let $w\in W(\lambda)=W(\mu)$.

(a) If $w\cdot\mu\in[\mu]^+$, then $w\cdot\lambda\in[\lambda]^+$. Also, when $w=w_{\mu'}$, for
 $\mu'\in[\mu]^+$, then $w=w_{\lambda'}$ for $\lambda'=w\cdot\mu$.

(b)   Assume $\lambda':=w\cdot \lambda\in[\lambda]^+$.  Put $\mu'=w\cdot\mu$. If $w_{\lambda'}=
w_{\mu'}$, then $\mu'\in[\mu]^+$.  \end{lem}

 \begin{proof} To prove (a), it must be shown first that if $\alpha\in\oPi$, then $\langle w\cdot\lambda,\alpha^\vee\rangle\geq 0$. By hypothesis, $0\leq\langle w\cdot\mu,\alpha^\vee\rangle= \langle w(\mu+\rho),\alpha^\vee)-1$ for any simple root $\alpha\in\oPi$. Hence, $(\mu+\rho,w^{-1}(\alpha)^\vee\rangle\geq 1$. 
 Since $\mu\in\sC^-$, this forces $w^{-1}(\alpha)<0$. Now suppose that $\langle w\cdot\lambda,\alpha^\vee\rangle <0$, for some $\alpha\oPi\subseteq\Phi_0(\lambda)$. Then, arguing as before, $\langle \lambda+\rho,w^{-1}(\alpha)^\vee\rangle \leq 0$. If $\langle \lambda+\rho,w^{-1}(\alpha)^\vee\rangle<0$, then, since $\lambda\in\sC^-$, $w^{-1}(\alpha)>0$,
 which is impossible. Thus, $w^{-1}(\alpha)\in\Phi_0(\lambda)\subseteq\Phi_0(\mu)$, again impossible.
 
 Next, for the last assertion of (a), notice that $w=w_{\mu'}$ means that $w_{\mu'}$ is of minimal length
 in its coset $wW_0(\mu)$, so it is certainly of minimal length in $wW_0(\lambda)\subseteq wW_0(\mu)$,
 giving $w=w_{\lambda'}$. Now (a) is completely established.
 
 Now we prove (b). We can assume that $w=w_{\mu'}=w_{\lambda'}$. Assume, for some $\alpha\in\oPi$,
 $\langle w\cdot\mu,\alpha^\vee\rangle<0$. In other words, $\langle w(\mu+\rho),\alpha^\vee\rangle<1$,
 i.~e., $\langle\mu+\rho, w^{-1}(\alpha)^\vee\rangle \leq 0$.  
 
 But $\langle w\cdot\lambda,\alpha^\vee\rangle\geq 0$, so that $\langle \lambda+\rho,w^{-1}(\alpha)^\vee)\geq 1$. Since $\lambda\in \sC^-$, this means that $w^{-1}(\alpha)<0$, so that $l(s_\alpha w)<l(w)$. Thus,
 $s_\alpha w\cdot\mu\not=w\cdot\mu$.   
Since $\mu\in\sC^-$, we have $\langle \mu+\rho,w^{-1}(\alpha)^\vee\rangle \geq 0$. 

Combining the two previous paragraphs, $\langle\mu+\rho,w^{-1}(\alpha)^\vee\rangle=0$. However, 
this means that $s_\alpha w(\mu+\rho)=\mu+\rho$, or equivalently $s_\alpha w\cdot w\cdot\mu$, a
contradiction. 
  \end{proof}  

\begin{prop}\label{important prop} Let $\lambda,\mu\in \sC^-_{\text{\rm rat}}$ satisfy  $\mu-\lambda\in WP^+$
and $\Phi_0(\lambda)\subseteq \Phi_0(\mu)$. Then the functor $T^\lambda_\mu$ takes objects in 
$\bO^+[\lambda]$ to objects in $\bO^+[\mu]$. Also, $\wT^\lambda_\mu$ maps objects in $\widetilde\bO^+[\lambda]$ to objects in $\widetilde\bO^+[\mu]$. If $\tau\in[\mu]^+$, then $L(w_{\tau}\cdot\lambda)$
is the unique irreducible module $L(\nu)\in\bO^+[\lambda]$ for which $T^\lambda_\mu L(\nu)=L(\tau)$. A similar statement
holds for $\widetilde L(\tau)$ and $\widetilde \bO^+[\lambda]$
\end{prop}
\begin{proof}We wish to show that $L':=T^\lambda_\mu L(\lambda')\in\bO^+[\mu]$. If $L'\not=0$, then
it is $L(\mu')$ where $\mu'=w\cdot\mu$ and $w_{\lambda'}=w_{\mu'}$ by Remark \ref{minimal length}.
Thus, $\mu'\in[\mu]^+$ by Lemma \ref{important lemma}(b), as desired. This proves the first 
assertion of the proposition. 
The second assertion concerning $\wT^\lambda_\mu$ is an easy consequence.

  The final assertion follows from Lemma \ref{important lemma}(a) and Remark \ref{minimal length}. 
\end{proof}
 
 We also obtain
 \begin{lem}\label{lemma6.6} Let $\lambda,\mu\in \sC^-_{\text{\rm rat}}$ be satisfy $\mu-\lambda\in WP^+$
and $\Phi_0(\lambda)\subseteq \Phi_0(\mu)$. For $w\in W(\lambda)$ with $w\cdot\lambda\in[\lambda]^+$,  
 $$\begin{cases}T^\lambda_\mu M(w\cdot\lambda)^+\cong M(w\cdot\mu)^+\\
 \widetilde T^\lambda_\mu\wM(w\cdot\lambda)^+\cong\wM(w\cdot\mu)^+.\end{cases}$$
 (If $w\cdot\mu\not\in[\lambda]^+$, then $M(w\cdot\mu)^+=0$ and $\widetilde M(w\cdot\mu)^+=0$.
 \end{lem}
 
 \begin{proof}By Propositions \ref{KTprop} and \ref{important prop}, 
 $T^\lambda_\mu M(w\cdot\lambda)^+$
 is a quotient of $M(w\cdot\mu)$ and belongs to $\bO^+[\mu]$. Thus, $T^\lambda_\mu M(w\cdot\lambda)^+$
 is a quotient of $M(\mu)^+$.  We can compare their characters using Proposition \ref{KTprop} and Weyl's character formula (applied for
 $\og$ on the irreducible modules $V(\overset\circ {w\cdot\lambda})$, $V(\overset\circ {w\cdot\mu})$).
 This shows (after inducing ${\mathfrak s}(w\cdot\lambda)$, ${\mathfrak s}(w\cdot\mu)$ from $\mathfrak p$
 to $\g$) that $T^\lambda_\mu M(w\cdot\lambda)^+$ and $M(w\cdot\mu)^+$ have the same
 character. Hence, they are isomorphic. A similar argument
 applies to show $\wT^\lambda_\mu\wM(w\cdot\lambda)^+\cong\wM(w\cdot\mu)^+$,
 proving the lemma.  
 \end{proof}
 
 Now the following analogue of Proposition \ref{important prop} follows as in the proof of the latter.
 
 \begin{prop}\label{prop6.7} Let $\lambda,\mu\in \sC^-_{\text{\rm rat}}$ satisfy the conditions (of Proposition \ref{KTprop}) $\mu-\lambda\in WP^+$
and $\Phi_0(\lambda)\subseteq \Phi_0(\mu)$.  Then
the following statements hold.

(a) $T^\lambda_\mu M^*(w\cdot\lambda)_+\cong M^*(w\cdot\mu)_+$ and $\wT^\lambda_\mu \widetilde M^*(w\cdot
\lambda)_+\cong\widetilde M^*(w\cdot\mu)_+$.

Also, when $w=w_{\mu'}$, for $\mu'\in[\mu]^+$, then $w=w_{\lambda'}$ for $\lambda'=w\cdot\mu$.

(b) If $\Phi_0(\lambda)=\Phi_0(\mu)$, then $T^\lambda_\mu$ and $\wT^\lambda_\mu$ 
give equivalences of categories 
$$\begin{cases}T^\lambda_\mu:\bO^{+,\infty}[\lambda]\overset\sim\to \bO^{+,\infty}[\mu]\\
\widetilde T^\lambda_\mu:\widetilde\bO^{+,\infty}[\lambda]\overset\sim\to\widetilde\bO^{+,\infty}[\mu].\end{cases}$$

(c) Again assume that $\Phi_0(\lambda)=\Phi_0(\mu)$. If $\Gamma\subseteq[\lambda]^+$, 
set $\Gamma'=\{w\cdot\mu\,|\,w\cdot\lambda\in\Gamma\}$.
Then $\Gamma$ is a poset ideal in $([\lambda]^+,\leq^W)$ if and only if $\Gamma'$ is a poset ideal in 
$([\mu]^+,\leq^W)$. In this case,  the posets $\Gamma$ and  $\Gamma'$ are isomorphic by the evident
map $w\cdot\lambda\mapsto w\cdot\mu$, and the functors $T^\lambda_\mu$
and $\wT^{\lambda}_\mu$ induce (by restriction) category equivalences 
$$\begin{cases}T^\lambda_\mu:\bO^{\Gamma,+}[\lambda]\overset\sim\to \bO^{\Gamma',+}[\mu]\\
\wT^\lambda_\mu:\widetilde\bO^{\Gamma,+}[\lambda]\overset\sim\to\widetilde\bO^{\Gamma',+}[\mu].
\end{cases}$$  
 \end{prop}
 
\section{Quantum enveloping algebras and category equivalences} We continue to work with the
indecomposable root system $\oPhi$, and we let $\ell$ be a positive integer. Set $D=(\theta_l,\theta_l)/(\theta_s,\theta_s)\in\{1,2,3\}$. 
Let 
\begin{equation}\label{defnofe}
e:=\begin{cases} \ell,\quad {\text{\rm if}}\,\,\ell\,\,{\text{\rm is odd}};\\
                        \ell/2,\quad{\text{\rm if}}\,\,\ell\,\,{\text{\rm is even}}.\end{cases}\end{equation}                  
                          There is a natural dot
action of the affine Weyl group $W_e=\overset\circ W\ltimes e\overset\circ Q$ on the set of integer weights $\oP\subseteq \fhstar$, given by $w\cdot\omu=
w(\omu+\orho)-\orho$ for $w\in W_e$. The action without the ``dot" $\cdot$ is the usual action
of $\overset\circ W$, and is translation on $e\overset\circ Q$. The fundamental reflections $s_0,s_1,\cdots, s_r$ for
$W_\ell$ consists of the usual reflections $s_i$ associated to fundamental roots $\alpha_i\in\overset\circ\Pi$,
$i=1,\cdots, r$, together with the reflections $s_0$ in the affine hyperplane $\{x\in\ofhstar\,|\,\langle x+\rho,\theta_s^\vee\rangle
=-e\,\}$. 

The following proposition is an easy calculation, similar to those given in \cite[p. 269]{T}. The
first observations of this kind are likely those of \cite{Kumar}. We state it only for $(D,e)=1$. (In
particular, this condition holds when $\oPhi$ is simply laced.)   A somewhat similar result holds 
without the assumption $(D,e)=1$, though the group $W_e$ must be modified; see \cite[Lemma 6.3]{T}. 
    
\begin{prop}\label{prop7.1} Let $\ell$ be a positive integer and let $e$ be as in (\ref{defnofe}).
Let $\lambda\in \sC^-_{\text{\rm rat}}$ with $\lambda(c)=k$, and assume
that $-(k+g)=\ell/2D$ and that $(e,D)=1$ with $e$ as above. There is an isomorphism $\phi_\ell:W_e\overset\sim\to W(\lambda)$ sending $s_0,\cdots, s_n$
to the fundamental reflections defined by $2\delta-\theta_s,\alpha_1,\cdots,\alpha_r\in\Phi(\lambda)$ 
if $\ell$ is odd, and to the fundamental reflections defined by $\delta-\theta_s,\alpha_1,\cdots,\alpha_r$
if $\ell$ is even.  In both cases,
$$\begin{cases}\phi_\ell(s_i)\cdot(\omu +k\chi)=w\cdot \omu+ k\chi,\quad w\in W_e;\\
\phi_\ell(\ell\gamma)\cdot(\omu+k\chi)=\omu+e\gamma+k\chi\,\, \text{\rm mod}\,\, {\mathbb Z}\,\delta,\,\gamma\in\overset\circ Q.
\end{cases}$$
In particular, if $\mu\in\fh^*$ with $\mu(c)=k$, then
$$\phi_\ell(w)\cdot\mu=w\cdot\omu+k\chi\,\quad\mod\,{\mathbb C}\delta.$$
More generally, if $-(k+g)=e/m$ for some positive integers $e,m$ with $(m,e)=1$ and $D|m$, then there is an
isomorphism $\phi:W_e\overset\sim\to W(\lambda)$, where $\phi(s_0)=s_{m\delta-\theta_s}$
and, for $i=1,\cdots, r$,  $\phi(s_{\alpha_i})$ is equal to the fundamental reflection defined by
$\alpha_1,\cdots,\alpha_r$.  The roots $m\delta-\theta_s, \alpha_1,\cdots, \alpha_r$ are the standard
(positive) fundamental roots in $\Phi(\lambda)$. 
\end{prop}

\begin{rem}The maps $\phi$ and $\phi_\ell$ agree when $\ell/2D=e/m$, which is our 
major interest in this paper (the ``quantum case", at least when 
$(D,e)=1$.) The exact description of $\phi_\ell(w)\cdot\mu$ above (or of $\phi(w)\cdot \mu$) is
$$ \phi_\ell(w)\cdot\mu=(w\cdot\omu)^{k,a}=(w\cdot\omu +k\chi)^{k,a}=w\cdot\omu+k\chi+b\delta,$$
where $b$ is chosen so that the Casimir operator $\Omega$ acts on $L(\phi_\ell(w)\cdot\mu)$ with the 
same action as on $L(\mu)$. That is, $a=(\mu+2\rho,\mu)$ and $b=\frac{a-(w\cdot\omu+2\orho,
w\cdot\omu)}{2(k+g)}$.  This reader is cautioned that  the projections onto ${\mathbb C}\delta$ for $\mu$ and for $\phi_\ell(w)\cdot\mu$
will generally be different. In particular, if $w=\ell\gamma$, $\gamma\in \overset\circ Q$, then $\phi_\ell(\ell\gamma)$ 
acts as a translation by $\ell\gamma$ mod ${\mathbb C}\delta$ on the elements $\mu$ of level $k$ in
$\fh^*$.  That is, $\phi_\ell(e\gamma)\cdot\mu=\mu+e\gamma$ $\mod\,{\mathbb C}\delta$. However, it is
not true in general that $\phi_\ell(e\gamma)\cdot\mu =\mu+e\gamma$ exactly, even if $\gamma$ is
replaced on the right by any fixed element of $\gamma+{\mathbb Z}\delta$.
 
One consequence of having to work mod ${\mathbb C}\delta$ with level $k$ weights is that the meaning
of dominance orders in the correspondence between $\oP$ and $\oP+k\chi$ $\mod {\mathbb C}\delta$
is lost. However, the Bruhat-Chevalley order is preserved. See the \S9, Appendix I for a discussion of the Bruhat-Chevalley orders  relative to the often used partial orders $\uparrow$ of strong linkage.
\end{rem}

Let $\zeta\in\mathbb C$ be a primitive
$\ell$th root of unity and set $q=\zeta^2$.   Let $U_\zeta=U_\zeta(\oPhi)$ be the (Lusztig) quantum enveloping algebra at $\zeta$ for the root system $\oPhi$. Let ${\mathcal Q}_\ell$ be the category of type 1, integrable, finite dimensional $U_\zeta$-modules.  According
to Tanisaki \cite[Thm. 7.1]{T}, summarizing work of Kazhdan-Lusztig \cite{KL} and Lusztig \cite{L}, there is
a category equivalence
\begin{equation}\label{KLequivalence} F_\ell:{\mathcal O}_{-\ell/2D-g}\overset{\sim}\to {\mathcal Q}_\ell\end{equation}
between the category ${\mathcal O}_{-\ell/2D-g}$ of $\wg$-modules  and the category ${\mathcal Q}_\ell$.  
This holds for all positive integers $\ell$, when $\oPhi$ has type $A$ or $D$, but restrictions are
required in the other cases; see \cite[Thm. 7.1 and Rem. 7.2]{T}---note that Rem. 7.2(a) should be replaced
by $r>h$, the Coxeter number.

  In the notation of the previous section, letting $k:=-\ell/2D-g$,
${\mathcal O}_k$
identifies with the category
 $$\bigcup_{\lambda\in\sC^-_{\text{\rm rat}}, \lambda(c)=k}\widetilde\bO^{\text{\rm finite}}[\lambda]^+.$$
In this notation, $\widetilde\bO[\lambda+b\delta]^+$ identifies with $\widetilde\bO[\lambda]^+$; indeed,
these are identical subcategories of the category of $\wg$-modules.
For $\mu\in[\lambda]^+$, $F_\ell\widetilde L(\mu)=L_\zeta(\overset\circ\mu).$  Also, as noted in \cite[Thm. 7]{T}, Kazhdan-Lusztig
prove that $F_\ell\widetilde M(\mu)\cong \Delta_\zeta(\overset\circ\mu)$.  (Here $\mu\equiv\omu+k\chi\,
\mod\,{\mathbb C}\delta$.)

  Let $\lambda\in \sC^-_{\text{\rm rat}}$,  and let $\Gamma$ be a finite ideal in $[\lambda]^+$ with
  respect to $\leq^{W(\lambda)}$ (see Remark \ref{categories}(b)). Recall that every object $M$ in the highest weight category $\widetilde\bO^{\Gamma,+}[\lambda]$ has a finite
  composition series.  Let $P$ be a projective generator for $\widetilde\bO^{\Gamma,+}[\lambda]$. If
  $A_\Gamma:=\End(P)$ , then   $\widetilde\bO^{\Gamma,+}[\lambda]\cong A_\Gamma$-mod.  It
  is convenient to choose $P$ so that the modules in $\bO^{\Gamma,+}[\lambda]$ are actually
  $A_\Gamma$-modules.
  Put $\gr A_\Gamma=\bigoplus \rad^nA_\Gamma/\rad^{n+1}A_\Gamma$ be the positively graded algebra obtained from $A_\Gamma$ using
  its radical filtration.\footnote{The algebra $A_\Gamma$ is only determined up to Morita equivalence.
  However, if $A$ and $B$ are two Morita equivalent finite dimensional algebras, then $\gr A$ is Morita equivalent to $\gr B$,
  and, in addition, these two graded algebras have equivalent graded module categories. See Appendix
  II.} Let $\gr\widetilde\bO^{\Gamma,+}[\lambda]$ be the category $\gr A_\Gamma$-grmod
  of graded $\gr A_\Gamma$-modules.

\begin{thm}\label{theorem7.3} Let $\lambda\in\sC^-_{\text{\rm rat}}$ be regular,
 and let $\Gamma$ be a finite ideal in the poset $([\lambda]^+,\leq^{W(\lambda)})$.

(a) The category $\gr \widetilde\bO^{\Gamma,+}[\lambda]= \gr A_\Gamma$-${\text{\rm grmod}}$ is a highest weight category with poset $\Gamma$ (or $\Gamma_\nat$) and standard objects $\gr M(\mu)^+$, $\mu\in\Gamma$.  

(b) The algebra
$\gr\,A_\Gamma$ is a Koszul algebra. Also, $ \gr \widetilde\bO^{\Gamma,+}[\lambda]$
has a graded Kazhdan-Lusztig theory with respect to the length function on $\Gamma$ defined by the Coxeter
length.  

(c) For $\mu=x\cdot\lambda\in[\lambda]^+$,  form the radical filtration $M(\mu)^+=F^0(\mu)\supsetneq F^1(\mu)\supsetneq\cdots
\supsetneq F^m(\mu)=0$ of $M(\mu)^+$. For $\nu=y\cdot\lambda\in[\lambda]^+$,
$[F^i(\mu)/F^{i+1}(\mu):L(\nu)]$ is the coefficient of $t^{l(x)-l(y)-i}$ in the inverse Kazhdan-Lusztig
polynomial $Q_{y,x}$ for the Coxeter group $W(\lambda)$.
\end{thm} 
 \begin{proof}
 Suppose that $\lambda\in\sC^-_{\text{\rm rat}}$ is regular and that $k=\lambda(c)$. Write 
 $-(k+g)=e/m$, where $(e,m)=1$. 
 
 \medskip\noindent
 \underline{Case 1:} $D$ divides $m$. Proposition \ref{prop7.1} gives an isomorphism $\phi:W_e\overset\sim\to W(\lambda)$ of Coxeter groups, matching up indicated sets of  fundamental reflections. For an integer $k'$, put  $\lambda'=\overset\circ\lambda+k'\chi$, so that $k'=\lambda'(c)$. We can choose $k'$ so that $\lambda'\in \sC^-_{\text{\rm rat}}$ is regular. We can also choose $k'$ so that $-(k'+g)=\ell'/2D$ for an integer $\ell'$ not
 divisible by $2$ or $3$ (if $\oPhi$ has type $G_2$). Defining $e'$ as in (\ref{defnofe}) (using $\ell'$ for
 $\ell$ and $e'$ for $e$), we have $(D,e')=1$.  Proposition 7.1 then gives an isomorphism $\phi:W_{e'}\overset\sim\to W(\lambda')$, again matching up fundamental reflections. Thus, there is an isomorphism
 $W(\lambda)\to W(\lambda')$ preserving fundamental reflections. Since $\lambda$ and $\lambda'$
 are both regular, $W_0(\lambda)=W_0(\lambda')$ is trivial. Therefore, by \cite[Thm. 11]{Fiebig}, there
 is a category equivalence
  $\bO[\lambda]\overset\sim\to\bO[\lambda']$.  (See Proposition \ref{ideals}.) Since the orders $\leq^{W(\lambda)}$
  and $\leq^{W(\lambda')}$ obviously correspond, standard modules correspond (implicit in \cite{Fiebig}).
  Similar comments apply to costandard modules. 
    
  The sets $[\lambda]^+$ and $[\lambda']^+$ are
  easily characterized (when $\lambda,\lambda'$ are regular) in terms of representing elements $w\cdot\lambda$, $w'\cdot\lambda'$
  by requiring $w$ to be of maximal length in $W_0w\subseteq W(\lambda)$. Thus, we get
  an equivalence $\bO^+[\lambda]\cong\bO^+[\lambda']$, as well as an equivalence of $\widetilde\bO^+[\lambda]\cong\widetilde\bO^+[\lambda']$. (See Proposition 5.5.) Since the partial orderings correspond, so do their
  ideals. Let $\Gamma'\leq[\lambda']^+$ correspond to $\Gamma\leq[\lambda]^+$. Clearly,
  $\widetilde\bO^{\Gamma,+}[\lambda]\cong\widetilde\bO^{\Gamma',+}[\lambda']$.
  
 Adjusting $k'$ further we can assume that $F_{\ell'}$ is a category equivalence. In particular, we
 can assume that $\ell'$ is odd, not divisible by 3 in type $G_2$, and $>h$. 
   If $\Gamma'$ is an ideal in the $\leq^{W(\lambda')}$ partial ordering, then $\overset\circ\Gamma'$ is an ideal in 
 the $\uparrow_{e'}$-order, by Appendix I, Theorem \ref{theorem9.7}.  Therefore, $\overset\circ\Gamma'$ is an ideal in the natural 
 order $\leq_\nat$ on $\oP$.  (As is well-known, when $\omu\leq_\nat \overset\circ\nu$, then $\omu\uparrow_{e'}
 \overset\circ\nu$.)  We have $F_\ell\widetilde\bO^{\Gamma',+}[\lambda])
 ={\mathcal Q}_\ell[\overset\circ\Gamma']$. Combining this with the previous paragraph gives
 an equivalence $\widetilde \bO^{\Gamma,+}[\lambda]\overset\sim\to{\mathcal Q}_{\ell'}[\overset\circ\Gamma']$
 of highest weight categories. Notice that by Proposition \ref{prop6.7}(c), the posets $\Gamma$ and $\Gamma'$
 are isomoprhic.  Now apply \cite[Cor. 8.5]{PS6}.  \footnote{Results claimed in \cite[\S9.5 (para. 1), Lem. 9.10.5, Thm. 9.10.2]{ABG}, together with the order compatibility
 result given in Theorem \ref{theorem96} in Appendix I, imply that $A_\Gamma$ itself is Koszul. (According to 
 \cite[Appendix]{HPS}, a result like our Theorem \ref{theorem9.6} is required in \cite[Lem. 9.10.5]{ABG} to make its
 proof work. In turn, this lemma is required for \cite[Thm. 9.10.1]{ABG}, a main result.) However, we only need
 that $\gr A_\Gamma$ is Koszul for the results below in the singular weight case, and 
  no better result is obtained in these cases by knowing $A_\Gamma$ is Koszul here. (Koszulity is not generally preserved under exact functors.)}  
   
 \medskip
 \noindent
 \underline{Case 2: $D$ does not divide $e$:} In this case, $W(\lambda)\cong W_e^\vee=\overset\circ W\ltimes Q^\vee$,
 which is the affine Weyl group for the dual root system $\Phi^\vee$. By \cite[Thm. 11]{Fiebig} again,
 $\widetilde \bO^{\Gamma,+}[\lambda]$ is equivalent to a similar category, but replacing $\mathfrak g$
 by the affine Lie algebra associated to the dual root system. Hence, Case 2 reduces to Case 1. (Notice that
 \cite[Thm. 11]{Fiebig} does not require the underlying Lie algebras to be the same!)
   \end{proof}

\begin{thm}\label{affineLiealgebramult} 
Let $\lambda=
\olambda+k\chi+b\delta\in\sC^-_{\text{\rm rat}}$. For $\mu\in [\lambda]^+$, $M(\mu)^+$ has a filtration
$M(\mu)^+=F^0(\mu)\supseteq F^1(\mu)\supseteq\cdots\supseteq F^m(\mu)=0$ in which each
section $F^i(\mu)/F^{i+1}(\mu)$ is a semisimple $\wg$-module, and such that, given
any $\nu\in[\lambda]^+$, the multiplicity
$[F^i(\mu)/F^{i+1}(\mu):L(\nu)]$ is the coefficient of $t^{l(w_\mu)-l(w_\nu)-i}$ in the inverse Kazhdan-Lusztig
polynomial $Q_{w_\nu,w_\mu}$ associated to $W(\lambda)$.
If $\lambda$ is regular, then the filtration $F^\bullet(\mu)$ is the radical filtration of $M(\mu)^+$.
\end{thm}
\begin{proof}  We prove this result in the ``quantum case" (discussed in the proof of Theorem \ref{theorem7.3}) in which
$-(k+g)=\ell/2D$ with $(D,e)=1$. We leave the ``non-quantum case" to the reader. The case in which $\lambda$ is regular is handled in
the previous theorem, so assume that 
  $\lambda$ is not regular. The weight $\lambda':=-2\orho +k\chi +b\delta$ is obviously regular (so $\Phi_0(\lambda')=\emptyset\subseteq\Phi_0(\lambda)$), and it lies in
 $\sC^-_{\text{\rm rat}}$. 
Also,  $\lambda-\lambda'=\olambda + 2\orho\in \overset\circ W\oPplus\subseteq WP^+$. 
In addition, $W(\lambda)=W(\lambda')$. 

Let $\mu\in[\lambda]^+$, so that $\mu=w_\mu\cdot\lambda$, where $w_\mu\in W(\lambda)$ has minimal
length among all $w\in W(\lambda)$ for which $\mu=w\cdot\lambda$. 
Thus, by Lemma \ref{important lemma}, $\mu':=w_\mu\cdot\lambda'\in[\lambda']^+$, and so
$T^{\lambda'}_\lambda M(\mu')^+=M(\mu)^+$. Put $F^i(\mu)=T^{\lambda'}_\lambda \rad^i M(\mu')^+$. Then $F^\bullet(\mu)$ is a filtration of $M(\mu)^+$. Since the functor $T^{\lambda'}_\lambda$ is exact we 
 have $$F^i(\mu)/F^{i+1}(\mu) \cong T^{\lambda'}_\lambda(\rad^iM(\mu')^+/\rad^{i+1}M(\mu')^+)$$ is
 semisimple. To determine the multiplicity $[F^i(\mu)/F^{i+1}(\mu):L(\nu)]$, we can
 assume that $\nu\in[\lambda]^+$. Write $\nu=w_\nu\cdot\lambda$.  Since
 $$[\rad^iM(\mu)^+/\rad^{i+1}M(\mu)^+:L(w_\lambda\cdot\lambda)] =[\rad^iM(\mu')^+/\rad^{i+1}M(\mu')^+:L(w_\nu\cdot\lambda')],$$
 which is the coefficient of $t^{l(w_\mu)-l(w_\nu)-i}$ in the
 inverse Kazhdan-Lusztig polynomial $Q_{w_\nu,w_\lambda}$ of $W(\lambda)$, the result follows.\end{proof}

Now we consider an analogous result for the quantum enveloping algebras $U_\zeta$ 
of type $A_n$ or $D_n$. Let $C^-$ be the anti-dominant chamber. Given a dominant weight
$\nu$, let $w_\nu\in W_e$ have minimal length so that $w_\nu^{-1}\nu\in\overline{C^-}$. 

 \begin{thm}\label{quantum}Assume that $\oPhi$ has type $A$ or $D$.  Also, for type $D_{2n+1}$, it is required that $e\geq 3$.\footnote{There is no other restriction on the positive integer $e$.  In case $e$ is odd, the arguments in Theorems \ref{theorem7.3}  and \ref{affineLiealgebramult} can be rearranged to 
 treat all quantum cases, using translation functors alone, without recourse to \cite{Fiebig}. Of course,
 use of \cite{Fiebig} not only handles the $e$ even case, but also allows a treatment
 for affine Lie algebras of all weights in ${\mathcal C}_{\text{\rm rat}}^-$ in Theorems \ref{theorem7.3} and 
 \ref{affineLiealgebramult}.}
 (a) Assume that $\mu,\nu\in\oPplus$ are dominant weights which $W_e$-conjugate.
 Then the standard module $\Delta_\zeta(\mu)$ has filtration $\Delta_\zeta(\mu)=F^0(\mu)\supseteq \cdots
 \supseteq F^m(\mu)=0$ by $U_\zeta$-submodules with each section $F^i(\mu)/F^{i+1}(\mu)$ a
 semisimple $U_\zeta$-module. Further, the multiplicity of $[F^i(\mu)/F^{i+1}(\mu):L_\zeta(\nu)]$ 
 can be taken to be the coefficient of $t^{l(w_\mu)-l(w_\nu)-i}$ in the inverse Kazhdan-Lusztig polynomial $Q_{w_\nu,w_\mu}$ associated to $W(\lambda)$. If $\mu$ is regular, then
 the filtration $F^\bullet(\mu)$ is the radical filtration of $\Delta_\zeta(\mu)$.
 
 (b) Assume that $e\geq h$. Let $\Gamma$ be a finite ideal of $e$-regular weights. Let $B_\Gamma$ be the finite
 dimensional algebra whose module category identifies with the category of $U_\zeta$-modules having
 highest weights in $\Gamma$.  Then the algebra $\gr B_\Gamma$ is a Koszul algebra. Also,
 the category $\gr B_\Gamma$-grmod has a graded Kazhdan-Lusztig theory with respect to
 the length function (defined on $W_e$-orbits in $\Gamma$ by the Coxeter length.
\end{thm}
 
 \begin{proof} This follows from Theorem \ref{affineLiealgebramult}, since for types $A$ and $D$ as
 indicated, the functor $F_\ell$ is an equivalence of categories, preserving standard modules.
 \end{proof}
 
   The argument above, traced through from the proof of Theorem \ref{theorem7.3}, gives the additional result
   that, under the hypotheses of Theorem \ref{quantum}(b), there is an
   isomorphism $\Ext^n_{B_\Gamma}(L,L')\cong \Ext^n_{\gr B_\Gamma}(L,L')$,  valid for all $n\geq 0$
   and for all irreducible
   $B_\Gamma$-modules $L,L'$ (which are naturally irreducible $\gr B_\Gamma$-modules);  see \cite[Cor. 8.5(a)]{PS6}. Consequently, the homological algebra of $B_\Gamma$ in Theorem
   \ref{quantum}(b) is very close to that of $\gr B_\Gamma$. Of course, $B_\Gamma$ is even isomorphic
   to $\gr B_\Gamma$, if we grant the Koszulity of $A_\Gamma$ argued in footnote 7.

\section{Applications} 

\numberwithin{equation}{subsection}
\numberwithin{thm}{subsection}

In this section, we reinterpret Theorem \ref{quantum} for the $q$-Schur algebras and then pass to a similar
result for Specht modules for Hecke algebras. Then we briefly raise some open questions. Finally, 
we obtain some similar results for classical Schur algebras in positive characteristic, involving the
James conjecture and the bipartite conjecture. 

\subsection{
$q$-Schur and Hecke algebras} Given a Coxeter system $(W,S)$, let $\wH=\wH(W)$ be the
Hecke algebra over $\sZ={\mathbb Z}[\q,\q^{-1}]$ (Laurent polynomials in a variable $\q$) with basis $\{\tau_w\,|\,w\in W\}$ and defining
relations  $$\tau_s\tau_w=\begin{cases} \tau_{sw},\,\,{\text{\rm if}}\,l(sw)=l(w)+1\\
\q\tau_w + (\q-1)\tau_{sw},\,\,{\text{\rm otherwise}}\end{cases}\,\,{\text{\rm for $s\in S,w\in W$. }}$$
Let $\Psi:\wH\to\wH$ be the $\sZ$-algebra involution defined by $\Psi(\tau_w)=(-\q)^{l(w)}\tau_{w^{-1}}^{-1}$.
If $\wM$ is a $\wH$-module, then $\wM^\Psi$ denotes the module obtained by making $\wH$ act
through $\Psi$.

For example, let ${\mathfrak S}_r$ be the symmetric group of degree $r$, and let $S=\{(1,2),\cdots,
(r-1,r)\}$.
Then $({\mathfrak S}_r,S)$ is a Coxeter system. 
  In this case, denote $\wH(W)$ simply by $\wH$, or $\wH(r)$ if $r$ needs to be mentioned.
Let $\Lambda(n,r)$ (resp., $\Lambda^+(n,r)$) be the set of compositions (resp., partitions) $\lambda$ of $r$ with at most $n$ parts; let $\Lambda(r)$ (resp., $\Lambda^+(r)$) be the set of all compositions (resp., partitions)
of $r$. For  $\lambda\in\Lambda(n,r)$, let $\wT_\lambda$ be the right ``permutation"
module for $\wH$ defined by $\lambda$, and 
$\wT(n,r):=\bigoplus_{\lambda\in\Lambda(n,r)}\wT_\lambda$.  The (integral) $\q$-Schur algebra (of bidegree $(n,r)$)  is the endomorphism algebra
\begin{equation}\label{Schur} \wS_\q(n,r):=\End_{\wH}(\wT(n,r)).\end{equation}
Given any commutative $\sZ$-algebra $K$ (e.~g., a field), let $\wS_\q (n,r)_K$ )(or just $S_\q (n,r)$ if $K$
is clear) denote the $K$-algebra
$\wS(n,r)\otimes K$---it has a description similar to (\ref{Schur}), replacing $\wH$ and $\wT(n,r)$ by
$H=\wH_K$ and $T(n,r)=\wT(n,r)_K$, respectively.

From now on assume that $K$ contains ${\mathbb Q}(\zeta)$, where $\zeta$ is a primitive $\ell$th
root of 1. Put $q=\zeta^2$, a primitive $e$th root of 1, in the notation of the previous section.  (No restriction
is placed on $e$, except as otherwise noted.) The
  triple $(S_q(n,r), T, H)$ satisfies the ``ATR" set-up prosyletized in \cite{CPS4}. In particular, given
$M\in{\text{\rm mod--}}H$ (right modules), put $M^\diamond:=\Hom_H(M,T)\in S_q(n,r){\text{\rm --mod}}$,
and, given $N\in S_q(n,r){\text{\rm --mod}}$, let $N^\diamond:=\Hom_{S_q(n,r)}(N,T)$. In this
way, there is a contravariant functor $M\mapsto M^\diamond$ (resp., $N\mapsto N^\diamond$)
from mod--$H$ to $S_q(n,r)$--mod (resp., $S_q(n,r)$--mod to mod--$H$). The convenience of denoting
them by the same symbol overcomes the annoyance of denoting them by the same symbol!

If $U_\zeta$ is the quantum enveloping algebra of type $A_{n-1}$ over $K$, there is a surjective homomorphism
$U_\zeta\twoheadrightarrow S_q(n,r)$. In this way, $S_q(n,r)$-mod is embedded in $U_\zeta$-mod. In
addition, $S_q(n,r)$--mod is a highest weight category with poset $(\Lambda^+(n,r),\trianglelefteq)$ 
defined by the dominance order on partitions. Irreducible modules $L_q(\lambda)$,
standard modules $\Delta_q(\lambda)$, and costandard modules $\nabla_q(\lambda)$ are all
indexed by $\Lambda^+(n,r)$. When regarded as $U_\zeta$-modules, $L_q(\lambda)$ gets
relabeled as $L_\zeta(\bar\lambda)$, where $\bar\lambda\in\oPplus$ is defined as follows: write
$\lambda=(\lambda_1,\cdots,\lambda_r)$, $\lambda_1\geq\cdots\geq\lambda_r$, and put
$\bar\lambda=a_1\varpi_1+\cdots+a_{r-1}\varpi_{r-1}$ with $a_i:=\lambda_i-\lambda_{i+1}$. (In
this expression, we label the simple roots for $A_{r-1}$ in the usual way, as in \cite{Bo}.) Each
$\lambda\in\Lambda^+(r)$, thus determines $w_\lambda\in{\mathfrak S}_r$ which has minimal
length among all $w$ satisfying $w^{-1}\cdot\bar\lambda\in C^-$ (the anti-dominant chamber
for $U_\zeta$).

In particular, for $\lambda\in \Lambda^+(n,r)$, 
we have 
\begin{equation}\label{modules}
\begin{cases} \Delta_q(\bar\lambda)^\diamond\cong S^\lambda,\\
\nabla_q(\bar\lambda)^\diamond\cong S_{\lambda'}^\Psi.\end{cases}
\end{equation}
In this expression, $\lambda'$ denotes the conjugate partition to $\lambda\in\Lambda^+(r)$.
In addition, the irreducible $H$-modules are indexed by the set $\Lambda^+(r)_{\text{\rm row-reg}}$ of
(row) $e$-regular partitions (i.~e., no row is repeated $e$-times). If $\lambda\in\Lambda^+(r)$, then $\bar\lambda$ is $e$-restricted (i.~e.,
it has all coefficients of fundamental dominant weights positive and $<e$) if and only if $\lambda'$
is $e$-regular. Then for $\lambda\in\Lambda^+_{\text{\rm res}}(r)$ (the $e$-restricted partitions),
\begin{equation}\label{irreducible modules}
L(\lambda)^\diamond\cong\begin{cases} D^\Psi_{\lambda'},
\quad \lambda\in\Lambda^+_{\text{res}}(r);\\
 0,\quad {\text{\rm otherwise}}.\end{cases}
 \end{equation} 
 
 \begin{thm}\label{Schur and Hecke} Assume that $K$ is a field containing ${\mathbb Q}(\zeta)$. 
 
 (a) For $\lambda\in\Lambda^+(n,r)$, the $q$-Weyl module $\Delta_q(\lambda)$ for the $q$-Schur algebra
 $S_q(n,r)$ has a filtration $\Delta_q(\lambda)=F^0(\lambda)\supseteq F^1(\lambda)\supseteq\cdots
 \supseteq F^m(\lambda)=0$ with semisimple sections $F^i(\lambda)/F^{i+1}(\lambda)$
 in which, given $\nu\in\Lambda^+(r)$, the multiplicity of $L_q(\nu)$ in $F^i(\lambda)/F^{i+1}(\lambda)$
 is the coefficient of $t^{l(w_{\bar\lambda})-l({w_{\bar\nu}})-i}$ in the inverse
 Kazhdan-Lusztig polynomial $Q_{w_{\bar\nu},w_{\bar\lambda}}$ associated to the affine Weyl
 group $W_e$ of type $A_{r-1}$.
 
(b) For $\lambda\in\Lambda^+(r)$, the $q$-Specht module $S^\lambda$ for the Hecke algebra $H$
has a filtration $0=G^0(\lambda)\subseteq G^1(\lambda)\subseteq\cdots
 \subseteq G^m(\lambda)=S^\lambda$ with semisimple sections $G^{i+1}(\lambda)/G^{i}(\lambda)$
 in which, given $\nu\in\Lambda_{\text{\rm res}}^+(r)$, the multiplicity of the irreducible $H$-module $D^\Psi_{\nu'}$
 in the section $G^{i+1}(\lambda)/G^{i}(\lambda)$ is the coefficient of $t^{l(w_{\bar\lambda})-l({w_{\bar\nu}})-i}$ in the inverse
 Kazhdan-Lusztig polynomial $Q_{w_{\bar\nu},w_{\bar\lambda}}$ associated to the affine Weyl
 group $W_e$ of type $A_{r-1}$.\end{thm}

\begin{proof} (a) is merely a translation into the language of $q$-Schur algebras of Theorem \ref{quantum}(a). 

As for
(b), we can take $n=r$. We first observe $T:=T(r,r)\cong S_q(n,r)f$ for an idempotent $f\in S_q(r,r)$
\cite[p. 664]{PS1}, and so $T$ is projective.  In addition, $T$ is a tilting module for
$S_q(r,r)$ and is therefore self-dual. See \cite[Thm. 8.4]{DPS1}. Thus, $T$ is also an injective
$S_q(r,r)$-module and so the ``diamond functor"
 $$(-)^\diamond=\Hom_{S_q(r,r)}(-,T): S_q(r,r){\text{\rm --mod}}\to
{\text{\rm mod--}}H$$
 is exact. Hence, (a) implies (b), putting $G^i(\lambda)=F^{m-i}(\lambda)^\diamond$.
\end{proof}

\subsection{Open questions} We raise some open questions.

\medskip
\begin{ques} Given $\lambda\in\Lambda^+(r)$, when is it true that the filtration described in the proof of Theorem \ref{Schur and Hecke}(b) is the socle
filtration of $S^\lambda$?  One should at least assume that $\lambda$ is restricted, and the case
where $\bar\lambda$ is regular in the sense of alcove geometry is already interesting.
\end{ques}

\begin{ques}\label{question two} When is there a positive grading on $H$ (with grade 0 semisimple) such
that for each
$\lambda\in\Lambda^+(r)$ there is a graded $H$-module structure
on $S^\lambda$, so that the multiplicities of irreducible $H_0\cong H/\rad H$-modules in each grade are as predicted by Theorem \ref{Schur and Hecke}(b)?   The same question may be asked for the quotient algebras $H(n,r)$ defined in
\cite{DPS1} and for the $H(n,r)$-modules $S^\lambda$, $\lambda\in\Lambda^+(n,r)$.
\end{ques}

\begin{ques}In  \cite{BKW}, a $\mathbb Z$-grading on Specht modules is given with
respect to a $\mathbb Z$-grading of the Hecke algebra. Since this grading is not, in general,
a positive grading with the grade 0 term a semisimple algebra, individual grades of a given graded module are not necessarily semisimple
modules.  Nevertheless, it appears from the form of the graded multiplicities in \cite{BKW}, together
with \cite{VV}, that these multiplicities are the same coefficients 
which appear in our Theorem \ref{Schur and Hecke}(b).  The question, therefore, arises as to when it is possible to ``regrade"
the Hecke algebra $H$ (shifting grades of projective indecomposable summands and passing to
an endomorphism algebra) to achieve a positively
graded algebra with grade 0 term semisimple in such a way that the induced regradings of the Specht modules agree with our filtration sections as in Question \ref{question two}. When this is possible, it answers Question \ref{question two} in a very specific way. \end{ques}

\begin{ques} For $\lambda\in\Lambda^+(n,r)$, when is the filtration for $\Delta_q(\lambda)$ described in Theorem \ref{Schur and Hecke}(a)
given by the radical series? The same question can be asked in all types; \cite[Thm. 8.4(c)]{PS6} gives a positive answer for regular highest weights. Lin \cite[Rem. 2.9(1)]{Lin} suggests a positive
answer in the singular case, at least  for generic weights.\end{ques}

\begin{ques}\label{question4} When is there a positive grading on $S_q(n,r)$ (with grade 0 semisimple) and on the
standard modules $\Delta_q(n,r)$ so that the grade $i$ section multiplicities are predicted by those in Theorem \ref{Schur and Hecke}(a) (for all $i$)? 
From the general theory of graded quasi-hereditary algebras \cite{PS6}, 
 if $S_q(n,r)$ has a positive grading, its standard modules will automatically be
graded. \end{ques}

\begin{ques}In \cite{Ariki}, Ariki gives a $\mathbb Z$-grading on $S_q(n,r)$
and the standard modules under mild restrictions on $e$. One can ask when some regrading process
in this case serves to give a positive question in Question \ref{question4} above. When $n\geq r$, \cite{Ariki} computes
the multiplicities of graded irreducible modules in his graded standard modules giving an answer
involving (inverse) Kazhdan-Lusztig polynomials. Is there some positive regrading of the grading in \cite{Ariki} possible
so that the  multiplicities in each grade agree with those in Theorem \ref{Schur and Hecke}(a)? \end{ques}

\begin{ques} When is $\gr S_q(n,r)$ a quasi-hereditary algebra? When is
it Koszul? One can also ask when $\gr S_q(n,r)$ has a Kazhdan-Lusztig theory in the
the sense of \cite{CPS1}, though it should be stated that the same question is open for singular
blocks of $S_q(n,r)$ itself.  \end{ques}
 
 Of course, all the above questions for $q$-Schur algebras can be asked in other types, i.~e., for generalized $q$-Schur
 algebras.

\subsection{Positive characteristic}  Now assume that $k$ is an algebraically closed field of
positive characteristic $p$. For positive integers $n,r$, let $S(n,r)=S_1(n,r)$ be the classical Schur algebra
over $k$ of bidegree $(n,r)$; see \cite{Green} for a detailed discussion in this special case. The irreducible $S(n,r)$-modules $L(\lambda)$ are indexed
by partitions $\lambda\in\Lambda^+(n,r)$. 

Form the PID $\sZ'={\mathbb Z}_{(p)}[\zeta]$, where 
$q:=\zeta^2=\sqrt[p]{1}$.
The $q$-Schur algebra $S_q(n,r)$, taken over
${\mathbb Q}(\zeta)$, has a standard integral $\sZ'$-form
$S_q(n,r)'$ such that $S(n,r)=S_q(n,r)'\otimes_{\sZ'}k$.
For $\lambda\in \Lambda^+(n,r)$, choose a $S(n,r)'$-stable $\sZ'$-lattice $L_q(\lambda)'$ in $L_q(\lambda)$, and
let $\overline{L_q(\lambda)}=L_q(\lambda)'\otimes_{\sZ'} k$ be the $S(n,r)$-module obtained by base change to $k$. The following is a
special case of a conjecture of James \cite{James}.

\begin{conj}\label{James} (James Conjecture, defining characteristic case) If $p^2>r$, $\overline{L_q(\lambda)}\cong L(\lambda)$
for all $\lambda\in\Lambda^+(n,r)$. \end{conj}

The (full) James conjecture has been verified for all $r\leq 10$ \cite{James}.\footnote{An advantage of using the James conjecture (over
the Lusztig conjecture) is that it does not require that $p\geq n$.} It is also known that the James conjecture holds, for a fixed $n$ and all $r$, provided that $p$ is sufficiently large.The conjecture is trivial unless $p\leq r$, so that $r$ must grow with $p$ for the conclusion to be substantive.

For $\lambda\in\Lambda^+(n,r)$, let $\bar\lambda$ be the dominant weight for $SL_n$ determined
by $\lambda$. Let $w_{\bar\lambda}$ be the unique element $x$ in the affine Weyl group $W_p$
such that $x\cdot\bar\lambda\in C^-$, the anti-dominant chamber. (See the discussion two paragraphs
above Theorem \ref{Schur and Hecke}.)
 In the same spirit, but motivated by \cite{FT}, \cite[Thm. 6.3]{CT}, and, especially, the notion of
 an abstract Kazhdan-Lusztig theory given in \cite{CPS1}, we conjecture
the following.

\begin{conj} \label{bipartite} (Schur Algebra Bipartite Conjecture, explicit form) For $p^2>r$, the $\Ext^1$-quiver of $S(n,r)$ is a bipartite graph.
Explicitly, the decomposition
$$E=\{\lambda\,|\, \,l(w_{\bar\lambda})\in 2{\mathbb Z}\},\quad O=\{\lambda\,|\, l(w_{\bar\lambda})\in 2{\mathbb Z}+1\,\}
$$
of $\Lambda^+(n,r)$ is compatible with the bipartite $\Ext^1$-quiver.
\end{conj} 

The first sentence  in Conjecture (\ref{bipartite})  means that $\Lambda^+(n,r)$ decomposes into a disjoint union $\Lambda^+(n,r)=E\,\,\cup \,\,O$
(the ``even" and the ``odd" partitions) such that 
$\Ext^1_{S(n,r)}(L(\lambda),L(\nu))\not=0$, whenever $\lambda,\nu\in\Lambda^+(n,r)$ are either both in $E$ or both
in $O$. The second sentence provides explicit $E$ and $O$.

 \begin{thm} Consider the Schur algebra $S(n,r)$ in positive characteristic $p$. Assume that $p^2>r$
 and that the James Conjecture \ref{James} and the Bipartite Conjecture \ref{bipartite} are true. 
 
 (a) For $\lambda\in\Lambda^+(n,r)$, let $\{F^i(\lambda)\}$ be the semisimple series of $\Delta_q(\lambda)$, $q=\zeta^2=\sqrt[p]{1}$, given in Theorem \ref{Schur and Hecke}. Let $\sL$ be an $S_q(n,r)'$-stable $\sZ'$-lattice in $\Delta_q(\lambda)$. Set
  $F^i(\lambda)'=F^i(\lambda)\cap\sL$ and 
 $\overline{F^i(\lambda)}=F^i(\lambda)'\otimes_{\sZ'}k$. Then
 $\{\overline{F^i(\lambda)}\}$ is a semisimple series in $\Delta(\lambda)$. Furthermore, for $\mu\in\Lambda^+(n,r)$,
 \begin{equation}\label{compositionfactors}
 [\overline{F^i(\lambda)}/\overline{F^{i+1}(\lambda)}:L(\mu)]=[F^i(\lambda)/F^{i+1}(\lambda):L_q(\mu)].
 \end{equation}
 In particular, this multiplicity is given by a coefficient in an inverse Kazhdan-Lusztig polynomial, as in
 Theorem \ref{Schur and Hecke}. 
 
 (b) For $\lambda\in\Lambda^+(r)$, the Specht module $S^\lambda$ for ${\mathfrak S}_r$ has
 a filtration $0=G^0(\lambda)\subseteq G^1(\lambda)\subseteq\cdots\subseteq G^m(\lambda)$, in
 which, given $\mu\in\Lambda^+_{\text{\rm res}}$, the multiplicity of the irreducible ${\mathfrak S}_r$-module
 $D_{\mu'}^\Psi$ in the section $G^{i+1}(\lambda)/G^i(\lambda)$ is the coefficient of
 $t^{l(w_{\bar\lambda})-l(w_{\bar\mu})-i}$ in the inverse Kazhdan-Lusztig polynomial $Q_{w_{\bar\mu},w_{\bar\lambda}}$ associated to the affine Weyl group $W_e$ of type $A_{r-1}$.
 \end{thm}
 
 \begin{proof} Since  Conjecture \ref{James} is assumed to hold, a section $F^i(\lambda)/F^{i+1}(\lambda)$
 reduces mod $p$ to an $S(n,r)$-module whose composition factor multiplicities are given by
 (\ref{compositionfactors}). To check that  $\{\overline{F^i(\lambda)}\}$ is a semisimple series  $\Delta(\lambda)$,
 it suffices to check that these sections are semisimple. But, if $L(\tau)$ and $L(\sigma)$ both appear
 in a composition series in $\overline{F^i(\lambda)}/\overline{F^{i+1}(\lambda)}$, then the coefficients of
 $t^{l(w_{\bar\lambda})-l(w_{\bar\tau})-i}$ in $Q_{w_{\bar\lambda},w_{\bar\tau}}$ and of $t^{l(w_{\bar\lambda})-l(w_{\bar\sigma})
 -i}$ in $Q_{w_{\bar\lambda},w_{\bar\tau}}$ are nonzero. However, these are polynomials in $\q=t^2$, so that
 $$l(w_{\bar\lambda})-l(w_{\bar\tau})-i \equiv l(w_{\bar\lambda})-l(w_{\bar\sigma})-i\,\,\mod 2.$$
 It follows that $w_{\bar\tau}$ and $w_{\bar\sigma}$ have the same parity, so that 
$\overline{F^i(\lambda)}/\overline{F^{i+1}(\lambda)}$ is semisimple by Conjecture \ref{bipartite}, as required for
(a).

Part (b) is proved in the same way as Theorem \ref{Schur and Hecke}(b).
\end{proof}

\begin{rem} James also conjectured \cite{James} a version of Conjecture \ref{James} for $q$-Schur algebras, with $q$ an
$e$th root of 1 and $ep>r$. It is at least reasonable to ask, as a question, if the analog of Conjecture
\ref{bipartite} holds under this assumption. Positive answers (to this question, and to this version of the
James conjecture), would have consequences for the finite general linear groups $G$ in a
non-defining characteristic $p$. See \cite[\S 9]{CPS8}, which 
shows the group algebra of $G$ has a large quotient which is a sum of tensor products of $q$-Schur  at various roots of unity $q$. Another reference is \cite{BDK}.   

For a different (though possibly related)  use of Weyl module filtrations
for (general) finite Chevalley groups, see \cite{Jantzenpaper} and \cite{Pillen}.
\end{rem}

\section{Appendix I} 
\numberwithin{equation}{section}
\numberwithin{thm}{section}

 If $\nu,\mu\in\sC$, write $\nu\uparrow\mu$ if $\nu=\mu$, or, if there are positive real
 roots $\beta_1,\cdots,\beta_m$ (allowing repetition) such that the corresponding reflections $s_{\beta_i}$
satisfy
\begin{equation}\label{strict}\nu=s_{\beta_m}s_{\beta_{m-1}}\cdots s_{\beta_1}\cdot\mu\leq \cdots \leq s_{\beta_1}\cdot\mu\leq\mu.\end{equation}
(Observe that this forces each $\beta_i\in\Phi(\nu)=\Phi(\mu)$.)
According to a result of Kac-Kazhdan (see \cite[Thm. 3.1]{KT2}), $\nu\uparrow\mu$ if and only if
$[M(\mu):L(\nu)]\not=0$.  Here is another equivalence:

\begin{prop}\label{First prop of appendix} Suppose that $\lambda\in\sC^-$ and that $\mu,\nu\in [\lambda]$.  Write
$\nu=y\cdot\lambda$ and $\mu=w\cdot\lambda$ with $y,w\in W(\lambda)$ of minimal length.
Then $\nu\uparrow\mu$ if and only if $y\leq w$ in the Bruhat-Chevalley order on $W(\lambda)$.
  \end{prop}
  
  \begin{proof} Suppose that $\nu\uparrow\mu$ and let $s_{\beta_1},\cdots,s_{\beta_m}$ be as above. We 
  can assume that all the inequalities in (\ref{strict}) are strict, i.~e., putting $\mu_0=\mu$ and, for
  $1\leq i\leq m$, $\mu_i=s_{\beta_i}\cdot\mu_{i-1}$ then $\mu_i-\mu_{i-1}=n_i\beta_i$ for some
  $n_i\in{\mathbb Z}^+$. In particular, $\beta_i\in\Phi^+(\lambda)$, for each $i=1,\cdots, m$.
  
  Write a reduced expression $w=t_1\cdots t_u$, where $t_j=s_{\alpha_j}$, for $\alpha_j$ a fundamental
  root of $\Phi(\lambda)$, $j=1,\cdots, u$, $u\in \mathbb N$. (Actually, $u\not=0$ since $y\cdot\lambda
  > \lambda$.)  A standard argument shows that $s_\beta w\cdot\lambda<w\cdot\lambda$ implies
  that  $\beta=t_1\cdots t_{j-1}(\alpha_j)$, for some $j$, $1\leq j\leq u$. Let $w_1$ be of minimal
  length with $w_1\cdot\lambda=s_\beta w\cdot\lambda$. That is, $w_1$ is a distinguished member of the left-coset $s_\beta wW(\lambda)$. Thus, $w_1\leq s_{\beta}w$, while $s_\beta w=t_1\cdots\widehat t_j\cdots
  t_m<w$. So $w_1<w$. Continuing this way, eventually gives $y<w$, as desired.
  
  Next, we start from the assumption $y<w$ and show that $\nu\uparrow \mu$ using induction on
  $l(w)-l(y)$. By \cite[Prop., p. 122]{Hump}, there exists $x\in W(\lambda)$ with $y<x<w$ and
  $l(w)-l(x)=1$. Put $x=w'v$, where $v\in W_0(\lambda)$ and $w'\in xW_0(\lambda)$ is the element
  of shortest length. Then $y=x'v'$ where $x'\leq x$ and $v'\leq v$. The minimality of $y$ implies
  that $y=x'\leq w'$. Possibly, $y=w'$, but, in any case, induction implies that $\nu\uparrow w'\cdot\lambda=
  x\cdot\lambda$. However, $x$ may be obtained from a reduced expression $w=t_1\cdots t_u$ by removing one of the fundamental reflections (with respect to $W(\lambda)$) reflections $t_j$.
  The discussion in the previous paragraph now shows that $x=s_\beta w$ for suitable $\beta\in\Phi^+(\lambda)$ with $s_\beta w\cdot\lambda<w\cdot\lambda$. Thus, $x\cdot\lambda\uparrow w\cdot\lambda\leq
  \mu$, so $v\uparrow x\cdot\lambda\uparrow\mu$, as desired. 
  \end{proof}

\begin{lem}\label{lemma9.2}
Suppose $\lambda \in \mathcal{C}^-_{\rm rat}$, $\mu \in [\lambda]^+$. Then there is a distinguished (i.e., minimal length) $(\overset\circ W, W_0(\lambda))$ double coset representative $d\in W(\lambda)$ with
$$
\mu = w_0 d \cdot \lambda.
$$
The element $d$ is unique; write $d = d(\mu)$.
\end{lem}

\begin{proof}
Choose an element $d \in W(\lambda)$ of minimal length with $d\cdot \lambda = w_0 \cdot \mu$. Since $\mu \in [\lambda]^+$, $s_\alpha \cdot \mu < \mu$ for all $\alpha \in \Phi^+$, whereas $s_\alpha \cdot (d\cdot \lambda)= s_\alpha(w_0 \cdot \mu) = s_\alpha(w_0(\mu+\zeta)-\rho) > w_0 \cdot \mu$. Thus $l(w_0d) \geq |\Phi^+l(d)$, by a count of separating  hyperplanes. However, $l(w_0d) \leq l(w_0)+l(d) = |\Phi^+|+l(d)$, and so necessarily $l(w_0d)=l(w_0)+l(d)$.

It follows $d$ has the form $d'w'$ where $d'$ is a distinguished double coset representation of $W_0dW(\lambda)$ and $w' \in W(\lambda)$. Minimality of $l(d)$ gives $d=d'$, as desired.
\end{proof}

In general, $w_0d$ will not be the element $\epsilon\in W(\lambda)$ of minimal length with $\epsilon \cdot \lambda = \mu$. We write $\epsilon = \epsilon(\mu)$ for such an element and note $d(\mu) \leq \epsilon(\mu) \leq w_0d(\mu)$.

\begin{prop}\label{prop9.3}
Suppose $\mu, \nu \in [\lambda]^+$ with $\lambda \in \sC^-_{\rm rat}$. The the following are equivalent:
\begin{itemize}
\item[(a)]
$d(\mu) \leq d(\nu)$;
\item[(b)]
$\epsilon(\mu) \leq \epsilon(\nu)$;
\item[(c)]
$w_0d(\mu) \leq w_0 d(\nu)$.
\end{itemize}

\end{prop}

\begin{proof}
Note $w_0d(\mu) = e(\mu)w'(\mu)$ for some $w'(\mu) \in W(\lambda)$. By definition, $e(\mu)$ is the element of minimal length in $e(\mu)W_0(\lambda)$.

Thus, if (c) holds, then $\epsilon(\mu) \leq w_0d(\mu) \leq w_0d(\nu) = \epsilon(\mu)w'(\nu)$, and it follows that $\epsilon(\mu) \leq d(\nu)$, since $u_0d(\nu)w'(\nu)^{-1}=w_1d(\nu)w_2$ with $w_1 \in W_0$, $w_2 \in W_0(\nu)$ and $l(w_1d(\nu)w_2)= l(w_1)l(d(\nu))+l(w_2)$. Then $d(\mu) = w'_1d(\nu)'w'_2$ with $w'_1 \leq w_1$, $d(\nu)' \leq d(\nu)$, $w'_2 \in w$, and $l(w'_1)+l(d(\nu)')+l(w'_2) = l(d(\mu))$. Minimality of $l(d(\mu))$ gives $d(\mu) = d(\nu)' \leq d(\nu)$, which is (a).

Clearly, (a) implies (c), and the proposition is proved.
\end{proof}

There is a further equivalence to add to the list. For $y,w \in W(\lambda)$, write $y \leq' w$, if $w_0yw_0 \leq w_0ww_0$, and put $l'(y) = l(w_0yw_0)$. In the above geometry, these operations amount to a change in generating fundamental reflections from the walls of the standard anti-dominant alcove to the walls of the standard dominant alcove.

Define $f(\mu)$, for $\mu \in [\lambda]^+$ and $\lambda \in \sC^-_{\rm rat}$, to be the element $f\in W(\lambda)$ with minimal $l'$-length satisfying $f\cdot(w_0 \cdot \lambda) = w_0 \cdot \mu$.

\begin{prop}\label{prop9.4}
Suppose $\mu \in [\lambda]^+$, $\lambda \in \sC^-_{\rm rat}$. Then $f(\mu) = w_0 d(\mu)w_0$.  Moreover, if also $\nu \in [\lambda]^+$, then
$$
f(\mu) \leq' f(\nu) \iff d(\mu) \leq d'(\nu).
$$
\end{prop}

\begin{proof}
By definition, $l(w_0dw_0)=l'(f)$. Thus, for $f=f(\mu)$, the element $w_0fw_0$ is the (unique) element $f'$ at minimal length $(f')$ satisfying $w_0f' \cdot \lambda = \mu$.  However, $w_0 d(\mu) \cdot \lambda = \mu$, so $w_0 f'W_0(\lambda) = w_0d(\lambda)W_0(\lambda)$.  Thus, $l(d(\mu)) \leq l(f')$, since $d(\mu)W_0(\lambda)$. Therefore, $d(\mu)=f'=w_0f(\mu)w$. So $f(\mu)=w_0d(\mu)w_0$.

This last assertion is now obvious, and the proof is complete.
\end{proof}

Finally, we need to compare the order $\leq'$ with the strong linkage order, in the sense
of \cite[II, \S6.7]{JanB}.  Sections 6.1--6.11 of the latter reference apply for any positive integer
$p$, which take to be $e$. For brevity,  we refer the reader to those sections for details on the alcove
notation used below. ``Alcoves" are regarded as certain open subsets of ${\mathbb R}\oP$, and the 
closure of an alcove is then a fundamental domain for the ``dot" action of the affine Weyl group
$W_e=\overset\circ W\ltimes e\overset\circ Q$. The ``standard alcove" $C$ satisfies
$$0<\langle x+\orho,\alpha^\vee\rangle<  e,$$
for all $x\in C$, and this inequality defines $C$. For any alcove $C_1$, there are unique integers
$n_\alpha=n_\alpha(C)$, for each $\alpha\in\Phi^+$, defined by
$$n_\alpha e<\langle x+\orho,\alpha^\vee\rangle<(n_\alpha+1) e$$
and the function $d(C_1)$ is defined as $\sum_{\alpha\in\Phi^+}n_\alpha$. Allowing the right
hand inequality ``$<(n_\alpha+1) e$" above to be the weaker ``$\leq (n_\alpha+1) e$, defines the
elements of the upper closure $\widehat C_1$ of $C_1$. There is a ``dot" action of $W_e$ on alcoves, agreeing with its ``dot" action on $\oP$, which is generated by reflections in the walls of $C$. For
$y\in W_ e$, write $l_ e(y)$ for the length of $y$ with respect to this set of generating reflections, and
$y\leq_ e w$ when $y,w\in W_ e$ and $y\leq w$ in the Bruhat-Chevalley order with respect to these
generating reflections.  If $\mu\in\oPplus$, define $f(\mu)=f_ e(\mu)$ to be the unique element
$f\in W_ e$ with $l_ e(f)$ minimal satisfying $f\cdot x=\mu$ for some $x\in\overline C$, the closure
of $C$. Equivalently, $\mu\in \widehat{f\cdot C}$ \cite[II, 6.11]{JanB}. From separating hyperplane
considerations,
\begin{equation}\label{star} l_ e(f)=d(f\cdot C).\end{equation}
Using this identity, the following lemma is mostly an easy exercise.

\begin{lem}\label{lemma9.5} Let $\xi \in\oPplus\cap\overset\circ Q$, i.~e., a dominant weight of $\og$ lying in the root lattice,  and let $z\in W_ e$ correspond to $e\xi\in e\overset\circ Q$
(i.~e., $z\cdot x=x+ e\xi$,
for all $x\in {\mathbb R}\oP$). Then, for any $\mu\in \oPplus$, 
$$\begin{cases}f(z\cdot\mu)=zf(\mu)\\ l_ e(zf(\mu))=l_ e(z) +l_ e(f(\mu)).\end{cases}$$

\end{lem}

\begin{proof} Note that $\widehat C_1+ e\xi=\widehat{C_1+ e\xi}$, just by the definition of the upper
closure. Thus, $f(z\cdot\mu)=zf(\mu)$. The length additivity is an easy calculations with the identity
(\ref{star}) and is left to the reader.
\end{proof}

Now we can prove our main result on strong linkage in the sense of \cite[II, 6.1--6.11]{JanB}. To avoid
conflict with our notation on $\fh^*$, we use $\uparrow_ e$ to denote the $\uparrow$ ordering in
\cite[II, Ch. 6]{JanB}.  That is, if $\mu,\nu\in
\oP$, write $\mu\uparrow_ e\nu$ to mean that $\mu=\nu$ or there exists a chain $\mu=\mu_0\leq
\mu_1\leq\cdots\leq\mu_m=\nu$ in $\oP$ and reflections $s_1,\cdots, s_m$ (not necessarily fundamental)
in $W_ e$ such that $s_i\cdot\mu_i=\mu_{i+1}$,
 for $i =1,\cdots, m$. The order $\leq $ used is the
usual dominance order on $\oP$.

\begin{thm}\label{theorem9.6} Let $\mu,\nu\in\oPplus$ lie in the same orbit of $W_ e$ under the dot action. Then
$$\mu\uparrow_ e\nu\iff f(\mu)\leq_ e f(\nu)$$
where $f(\mu)$ and $f(\nu)$ are the elements of $W_ e$
described above. 
\end{thm}

\begin{proof} First, note the general Coxeter group fact that if $u,w,z\in W_ e$, and if the
lengths of $zy$ and of $zw$ are obtained by adding the length of $z$ to that of $y$ and to that
of $w$, respectively, then
$$y\leq_e w\iff zy\leq_e zw.$$
The implication that $y\leq_ e w\implies zy\leq_ e zw$ is obvious and the reverse
implication reduces immediately to the case $ e(z)=1$, where it is obvious. (If $zy\leq_ e zw$, then $y\leq_ e zy\leq_ e w$; otherwise, $zy=zw'$, where $w'\leq_ e w$, whence $y=w'\leq_ e w$.)

This fact, together with the preceding lemma,  allows us to replace $y=f(\mu)$ and $w=f(\nu)$ by $zy,zw$ with $z\in W_ e$
corresponding to an element $ e\xi$ with $\xi\in\overset\circ Q$ and also dominant (i.~e., in $\oPplus$).
At the same time, we can replace $\mu,\nu$ by $\mu+e\xi, \nu+e\xi$, respectively. Obviously
$\mu\uparrow_e\nu\iff\mu+e\xi\uparrow_e\nu+e\xi$. This equivalence holds for arbitrary
weights $\mu,\nu\in\oP$ and thus may also be applied to intermediate instances of $\uparrow_e$.

So starting from $\mu\uparrow_ e \nu$ and making such an adjustment, we may assume all elements
$\mu_0\leq\mu_1\leq\cdots\leq \mu_m$ in the defining chain are dominant, and also $l_ e(f(\mu_i))=
l_ e(f(\mu_{i-1}))+1$, for $i=1,\cdots, m$. (See the argument in \cite[II, 6.10]{JanB}.) This implies that
$\mu_{i-1}$ is obtained from $\mu_i$ by striking out a simple reflection, and so, in particular,
$f(\mu_{i-1})<f(\mu_i)$ for each $i=1,\cdots, m$, and $f(\mu)\leq_ e f(\nu)$.  

Next, suppose that $f(\mu)\leq_ e f(\nu)$ with $\mu,\nu\in\oPplus$ in the same $W_ e$-orbit
under the dot action. We want to show that $\mu\uparrow_ e \nu$. By construction, $\mu,\nu$ belong 
to $\widehat{f(\mu)\cdot C}$ and $\widehat {f(\nu)\cdot C}$, respectively. This is equivalent to $f(\mu)\cdot
C\uparrow_ e f(\nu)\cdot C$, adapting the notation of \cite[II, 6.11(3)]{JanB}. As in \cite{JanB}, we 
say that an alcove $C_1$ is dominant if $n_\alpha(C_1)\geq 0$ for all $\alpha\in\oPhi^+$. Both
$f(\mu)\cdot C$ and $f(\nu)\cdot C$ are dominant. So, it suffices to prove that $y_1\cdot C\uparrow_ e 
y_2\cdot C$, whenever $y_1\leq y_2$ and $y_1\cdot C, y_2\cdot C$ are both dominant ($y_1,y_2\in W_ e)$.

We note that $y\cdot C$ is dominant if and only if $y$ is of minimal
length $l_ e(y)$ in the coset
$\overset\circ W y$. 

We now proceed by induction on the difference $m=l_ e(y_0)-l_ e(y_1)$. Without loss, $m\not=0$
(where the desired result is trivially true). By \cite[Prop. p. 122]{Hump}, there exists $y\in W_ e$ with $y_1\leq_e
y<_e y_2$ and $l_ e(y)+1=l_ e(y_2)$. Thus, $y=sy_2$, where $s$ is an (affine) reflection in the
hyperplane $H=H_{\alpha,n}=\{x\in{\mathbb R}\oP\,|\,\langle x+\orho,\alpha^\vee\rangle=n e\}$.
The hyperplane $H$ must be one of those separating the dominant alcove $y\cdot C$ from $C$. See
\cite[Thm., p. 93; Thm. p 58, Ex. 1, p. 58]{Hump}. Since $y\cdot C$ is dominant,  $n>0$ for
the parameter $n$ in $H=H_{\alpha, n}$.

  Now \cite[II, 6.8 Prop.]{JanB} can be applied, to give a (unique) $w\in\overset\circ W$ with $wy\cdot C$
  dominant and $wy\cdot C\uparrow y_2\cdot C$. (\cite{JanB} actually proves much more.) We have $y_1\leq_e y$, and so $y_1\leq_e wy$, since $wy$ is of minimal length in the coset $\overset\circ Wy$.
  (We use that $y=w^{-1}(wy)$ with $l_e(y)=l_e(w^{-1})+l_e(wy)$. Thus a reduced expression of $y_1$ can be
  obtained from one for $y$ by omitting suitable reflections from $w^{-1}$, and from $wy$. However,
  $y_1$ is minimal in $\overset\circ W y_1$, so $y_1\leq_e wy$.)
    We have 
  $l_ e(y_1)\leq l_ e(wy)\leq l_ e(y)< l_ e(y_2)$, so $y_1\cdot C\uparrow wy\cdot C$ by 
  induction. It follows that $y_1\cdot C\uparrow y_2\cdot C$, and the theorem is proved.\end{proof}
  
  We can now give the promised relation of the order $\leq'$ with the strong linkage order $\uparrow_e$.
  The order $\leq'$ is on $W(\lambda)$, $\lambda\in\sC^-_{\text{\rm rat}}$,  whereas $\uparrow_e$ is on $\oP$. We have already compared
  $\uparrow_e$ on $\oP$ with $\leq_e$ on $W_e$ in Theorem \ref{theorem9.6}. It remains now only to compare $\leq_e$ on
  $W_e$ with $\leq'$ on $W(\lambda)$. For simplicity, we state the comparison only in the simply laced case, through
  a similar result holds in all types. 
  
  \begin{thm}\label{theorem9.7} Assume $\oPhi$ is simply laced and let $\ell$ be a positive integer, $e=\ell$ if
  $\ell$ is odd and $e=\ell/2$ if $\ell$ is even. 
For $y,w \in W_e$, $y\leq_e w$ if and only if $\phi_\ell (y) \leq' \phi_\ell (w)$. Also,
  
  (a) If $\lambda\in C$ is dominant, and $y,w\in W_e$
  are minimal with $y\cdot\lambda$, $w\cdot\lambda$ dominant, then
  $y\leq_e w $ if and only if $y\cdot\lambda\uparrow_e w\cdot\lambda$.
  
  (b)  In the affine case,
  let $\lambda\in\sC^-_{\text{\rm rat}}$ and let $\mu,\nu\in [\lambda]$. Write  $\mu=y\cdot\lambda$ and
  $\nu=w\cdot\lambda$, with $y':=w_0y w_0$ and $w'=w_0yw_0$ both minimal with respect
  to $\leq'$ (or, equivalently, $y,w$ are minimal with respect to $\leq$). Then $y'\leq'w'$ if and only if $\mu\uparrow\nu$. (Also, equivalent is $y\leq w$.)
\end{thm}

\section{Appendix, II: Graded algebras and Morita equivalence} We show that if $A$ is a finite
dimensional algebra which is Morita equivalent to another finite dimensional algebra $B$, then
the graded algebras $\gr\, A$ and $\gr\,B$ obtained from the radical filtrations of $A$ and $B$ 
are also Morita equivalent. Also, the corresponding categories of finite dimensional graded 
modules categories are equivalent.

        Let $A$ be a finite dimensional algebra over a field. Let $P$ be any (finite dimensional) projective $A$-module.     
        Any $A$-map $x: P\to P$ which doesn't send
$P$ to $\rad\,P$ can be multiplied by another $A$-map $P\to P$
to get the identity on an irreducible summand of
$P/\rad\,P$. So the original map $x$ can't be in the radical of $\End\,P$.
On the other hand, if the original map $x$ does send $P$ into
$\rad\,P$, so does any element in the ideal $N$ that $x$ generates in
$\End P$, So a power of the ideal $N$ is zero. Consequently, $x$ must be
in $\rad\,\End\,P$. We have now characterized the radical $\rad\,\End\, P$ of $\End\,P$
as the space of all maps
$x:P\to P$ with $xP$ contained in $\rad \,P$. 

Next, letÕs add the condition that $P$ is a projective generator. Then $\rad P$ is clearly
the span of images $xP$ of elements $x$ in $\End\,P$ such that $xP$ is contained in
$\rad\,P$.  That is, $\rad\,P =(\rad\,\End\,P)P$. It follows inductively that $\rad^rP=(\rad\End\,P)^rP$.
(Assume this isomorphism for $r-1$ and multiply both sides by $\rad\, A$.) Thus, when $P$
is viewed as a left $A^{\prime\prime}=\End\,P$-module, and we make it into a 
$\gr A^{\prime\prime}$-module---call it
$\gr^{\prime\prime}P$---by using the radical series of $A^{\prime\prime}$, we get a module natural identification of $\gr^{\prime\prime}P$
with $\gr\, P$ as a vector space. If we let $A'$ be the opposite algebra of $ A^{\prime\prime}$, then the vector
space $\gr\, P$ provides a $\gr\, A, \gr \,A'$-bimodule through this identification.
 
However, Morita theory tells us that, if we similarly regard $P$ as an $(A,A')$-bimodule,
it provides a Morita context. In particular, $P$ is a projective generator as right $A'$-module
or left $A^{\prime\prime}$-module. Thus $\gr P$ is a projective generator is a left $\gr\, A^{\prime\prime}$ or right $\gr\, A'$-module.
We have thus recaptured the Morita context provided by $P$ for $A$ and $A'$ by one
provided by $\gr\, P$ for $\gr\, A$ and $\gr\, A'$, and so the latter algebras are Morita equivalent.
Moreover, since the bimodule providing this equivalence is graded,  we obtain
an equivalence between the categories of graded $\gr\, A$-modules and graded $\gr\, A'$-modules.
In particular $\gr A$ is Koszul and only if $\gr\, AÕ$ is Koszul.

\end{document}